\documentclass[11pt]{amsart}
\usepackage{amsmath,amssymb,amsfonts}
\usepackage{enumerate}
\usepackage{tikz-cd}
\usepackage{url}
\usepackage{hyperref}
\hypersetup{colorlinks=true, linktoc=all, linkcolor=blue}
\usepackage{todonotes}

\newtheorem{theorem}{Theorem}[section]

\theoremstyle{definition}
\newtheorem{definition}[theorem]{Definition}
\newtheorem{proposition}[theorem]{Proposition}
\newtheorem{lemma}[theorem]{Lemma}
\newtheorem{corollary}[theorem]{Corollary}
\newtheorem{construction}[theorem]{Construction}
\newtheorem{notation}[theorem]{Notation}
\newtheorem{example}[theorem]{Example}
\newtheorem{remark}[theorem]{Remark}

\newcommand{\Ch}{\mathrm{Ch}}
\newcommand{\Ext}{\mathrm{Ext}}
\newcommand{\Gr}{\mathrm{Gr}}
\newcommand{\Hom}{\mathrm{Hom}}
\newcommand{\id}{\mathrm{id}}
\newcommand{\Mod}{\mathrm{Mod}}

\newcommand{\Ob}{\mathrm{Ob}}
\newcommand{\op}{\mathrm{op}}
\newcommand{\Sq}{\mathrm{Sq}}

\begin{document}

\title{Unstable Modules with the Top $k$ Squares}
\author{Zhulin Li}
\address{Department of Mathematics, Massachusetts Institute of Technology, Cambridge, MA, USA}
\email{zhulin@mit.edu}

\maketitle

\begin{abstract}
Unstable modules over the Steenrod algebra with only the top $k$ operations are introduced in the language of ringoids.
We prove the category of such modules has homological dimension at most $k$.
A pratical method, which generalizes the $\Lambda$ complex, to compute the $\Ext$ group from such modules to spheres is proposed.
We are also able to establish several functors to relate such modules and unstable modules over the Steenrod algebra, and to describe the connections between the $\Ext$ groups in them.
\end{abstract}

\tableofcontents

\section*{Introduction}
Let $A$ be the Steenrod algebra over the field $\mathbb{F}_{2}$.
The purpose of this paper is to investigate the category $\mathcal{U}_{k}$ of unstable left $A$-modules where only the top $k$ Steenrod squares are allowed.
In general, on an homogeneous element of degree $n$ the top $k$ Steenrod squares on it are $\Sq^{n},\Sq^{n-1},\dots,\Sq^{n-k+1}$.

The cohomology of a topological space is naturally an unstable left $A$-module.
The category $\mathcal{U}$ of such modules has been studied extensively (see e.g. \cite{schwartz1994}).
It is a basic problem in algebraic topology to compute the $\Ext$ groups between spheres in this category, as they coincide with the $E_{2}$ page of the unstable Adams spectral sequence.
Unfortunately, those $\Ext$ groups are often difficult to compute and the category $\mathcal{U}$ is not of finite homological dimension.
Our category $\mathcal{U}_{k}$ is easier to work with than $\mathcal{U}$ in the sense that it is of finite homological dimension.
Computation of $\Ext$ groups in $\mathcal{U}_{k}$ in turn contributes to computation of $\Ext$ groups in $\mathcal{U}$, because when $N$, a module in $\mathcal{U}$, is bounded above degree $n$ and $k$ is large compared to $n$, the $\Ext$ groups into $N$ in those two categories agree with each other.

Just as the $\Lambda$ complex computes the $\Ext$ groups into spheres in $\mathcal{U}$, our $\Lambda_{k}$ complex gives an algorithmic method to compute the $\Ext$ groups into spheres in $\mathcal{U}_{k}$.
Our arguments also give an alternative proof to the fact that the traditional $\Lambda$ complex computes the $\Ext$ groups into spheres in $\mathcal{U}$.

The paper consists of 6 sections, the first of which introduces the theory of ringoids and prepares for the formal definition of category $\mathcal{U}_{k}$ in Section \ref{sec:U_k}.
Some examples of modules in $\mathcal{U}_{k}$ and the symmetric monoidal structure of $\mathcal{U}_{k}$ are also given in Section \ref{sec:U_k}.
Section \ref{sec:functors} is devoted to various functors between $\mathcal{U}_{k},\mathcal{U}_{k+1},\mathcal{U}$.
Our main theorem on the homological dimension of $\mathcal{U}_{k}$ is presented in Section \ref{sec:U_k_homodim}.
Section \ref{sec:Lambda_k} introduces the $\Lambda_{k}$-complex of a module in $\mathcal{U}_{k}$ and shows that it computes the $\Ext$ group from that module into spheres.
The convergence of an inverse system of $\Ext$ groups in $\mathcal{U}_{k}$'s to the $\Ext$ group in $\mathcal{U}$ is studied in Section \ref{sec:invsys}.

The author would like to thank Professor Haynes Miller without whom this paper would not exist.
Professor Miller's conjecture that the homological dimension of $\mathcal{U}_{k}$ is at most $k$, was the starting point of this project.
Professor Miller pointed the author to this topic and the author benefited a lot from weekly discussions with him.

\section{Ringoids}
\label{sec:ringoids}
In this section, we briefly review the theory of ringoids, which is not a novelty of this paper.
Informally, a ringoid is a ring with several objects.
Similar to rings, ringoids have subringoids, ideals and quotients.
We point interested readers to more extensive literature e.g. \cite{mitchell1972}.

\subsection{Ringoids and modules}
\begin{definition}[Preadditive category]
A category $\mathcal{A}$ is called preadditive if each morphism set $\mathcal{A}(x, y)$ is endowed with the structure of an abelian group in such a way that the compositions \[\mathcal{A}(x, y)\times\mathcal{A}(y, z)\to\mathcal{A}(x, z)\] are bilinear.
\end{definition}

\begin{definition}[Additive functor]
A functor $F:\mathcal{A}\to\mathcal{B}$ of preadditive categories is called additive if and only if \[F:\mathcal{A}(x, y)\to\mathcal{B}(F(x), F(y))\] is a homomorphism of abelian groups for all objects $x$ and $y$ in $\mathcal{A}$.
\end{definition}

\begin{definition}[Ringoid]
A ringoid is a small preadditive category and a morphism of ringoids is an additive functor.
Denote the category of ringoids by $\mathbf{Ringoid}$.
A ring is just a special ringoid --- a ringoid with a single object.
Many statements in ring theory can be generalized to ringoids.
\end{definition}

\begin{definition}[Left module over a ringoid]
\label{def:moduleRingoid}
If $\mathcal{A}$ is a ringoid, then a left $\mathcal{A}$-module is a covariant additive functor from $\mathcal{A}$ to the category $\mathbf{Ab}$ of abelian groups.
These modules form a category $\mathcal{A}\mathbf{Mod}$.
From now on, by an $\mathcal{A}$-module we mean a left $\mathcal{A}$-module.
Note that for any object $x$ in a ringoid $\mathcal{A}$, the covariant functor $\mathcal{A}(x,-):\mathcal{A}\to\mathbf{Ab}$ is an $\mathcal{A}$-module.
\end{definition}

\begin{proposition}
\label{prop:yoneda}
For any object $x$ in a ringoid $\mathcal{A}$ and any $\mathcal{A}$-module $M$, we have the following isomorphism of abelian groups
\[\mathcal{A}\mathbf{Mod}(\mathcal{A}(x,-),M) \cong M(x).\]
What's more, this isomorphism is natural in $M$, i.e. given an $\mathcal{A}$-module map $f:M\to N$, we have the following commutative diagram
\[\begin{tikzcd}
\mathcal{A}\mathbf{Mod}(\mathcal{A}(x,-),M)\arrow[r]\arrow[d] &M(x)\arrow[d]\\
\mathcal{A}\mathbf{Mod}(\mathcal{A}(x,-),N)\arrow[r] &N(x).
\end{tikzcd}\]
That is to say, the functor $\mathcal{A}\mathbf{Mod}(\mathcal{A}(x,-), -)$ from $\mathcal{A}\mathbf{Mod}$ to $\mathbf{Ab}$ is equal to the functor sending $M$ to $M(x)$.
\end{proposition}
\begin{proof}
This follows from the Yoneda lemma.
\end{proof}

\begin{proposition}
For any object $x$ in the ringoid $\mathcal{A}$, the $\mathcal{A}$-module $\mathcal{A}(x,-)$ is projective.
\end{proposition}
\begin{proof}
By Proposition \ref{prop:yoneda}, the functor $\mathcal{A}\Mod(\mathcal{A}(x,-),-)$ is equal to the functor $\mathcal{A}\Mod\to\mathbf{Ab}$ sending $M$ to $M(x)$.
So this functor is exact and the module $\mathcal{A}(x,-)$ is projective.
\end{proof}

\begin{proposition}
\label{prop:enoughProj}
$\mathcal{A}\mathbf{Mod}$ is a complete and cocomplete abelian category with enough projectives.
\end{proposition}
\begin{proof}
Since $\mathbf{Ab}$ is a complete and cocomplete abelian category which satisfies \emph{Ab}5, the same is true of $\mathcal{A}\mathbf{Mod}$.
See \cite{grothendieck1957} for more details on the axiom \emph{Ab}5.

To see that $\mathcal{A}\mathbf{Mod}$ has enough projectives, it suffices to construct for every $\mathcal{A}$-module $M$ an epimorphism $e:F\to M$ with $F$ projective.
For each object $x$ in the ringoid $\mathcal{A}$ and each element $m$ in the abelian group $M(x)$, we take the $\mathcal{A}$-module map $\mathcal{A}(x,-)\to M$ corresponding to $m\in M(x)$, according to Proposition \ref{prop:yoneda}.
Then the map $\mathcal{A}(x,-)\to M$ sends $\id_{x}\in\mathcal{A}(x,x)$ to $m\in M(x)$.
Taking the direct sum of all such maps, we get an epimorphism from a direct sum of projective modules to $M$.
\end{proof}

\begin{proposition}
\label{prop:projFree}
An $\mathcal{A}$-module $M$ is projective if and only if there exists another $\mathcal{A}$-module $N$ such that \[M\oplus N = \bigoplus_{i\in I}\mathcal{A}(x(i),-)\] for some index set $I$.
\end{proposition}
\begin{proof}
If $M\oplus N = \bigoplus_{i\in I}\mathcal{A}(x(i),-)$, then $M\oplus N$ is projective and $M$, as a retract of $M\oplus N$, is projective as well.

For the other direction, let $M$ be a projective $\mathcal{A}$-module.
According to the proof to Proposition \ref{prop:enoughProj}, we can construct an epimorphism $e:F\to M$ with $F=\bigoplus_{i\in I}\mathcal{A}(x(i),-)$.
Since $M$ is projective, we can construct $f:M\to F$ such that $e\circ f$ is equal to the identity map on $M$.
Denote the kernel of $e$ by $N$.
Then it is easy to check that the map $M\oplus N\to F$ induced by $f:M\to F$ and the monomorphism $N\to F$ is an isomorphism.
\end{proof}

\subsection{Subringoids and ideals}
A subring of a ring is an abelian subgroup that is closed under multiplication and contains the multiplicative identity of the ring.
We generalize this to the following definition of subringoid.
\begin{definition}[Subringoid]
A subringoid is a wide preadditive subcategory of the ringoid.
In other words, a subringoid $\mathcal{B}$ of a ringoid $\mathcal{A}$ is a subcategory of $\mathcal{A}$ such that
\begin{itemize}
\item
$\Ob(\mathcal{B}) = \Ob(\mathcal{A})$,
\item
$\mathcal{B}(x, y)$ is an abelian subgroup of $\mathcal{A}(x, y)$ for all objects $x, y$.
\end{itemize}
\end{definition}

An ideal of a ring is an abelian subgroup that is closed under multiplication with elements in the ring both from the left and right.
We generalize this to the following definition of an ideal of a ringoid.
\begin{definition}[Ideal and quotient]
An ideal $\mathcal{I}$ of a ringoid $\mathcal{A}$ consists of an abelian subgroup $\mathcal{I}(x,y)$ of $\mathcal{A}(x,y)$ for each pair of objects $x,y\in\Ob(\mathcal{A})$ such that for all objects $x,y,z\in\Ob(\mathcal{A})$,
\begin{itemize}
\item
the image of composition $\mathcal{I}(x, y)\times \mathcal{A}(y, z)\to \mathcal{A}(x, z)$ lies in $\mathcal{I}(x, z)$,
\item
the image of composition $\mathcal{A}(x, y)\times \mathcal{I}(y, z)\to \mathcal{A}(x, z)$ lies in $\mathcal{I}(x, z)$.
\end{itemize}
Given a ringoid $\mathcal{A}$ and an ideal $\mathcal{I}$ in it, we can form the quotient $\mathcal{Q} = \mathcal{A}/\mathcal{I}$ by equalizing morphisms in $\mathcal{I}$ to zero.
We call $\mathcal{Q}$ the quotient ringoid of the ringoid $\mathcal{A}$ by the ideal $\mathcal{I}$.
More precisely, given a ringoid $\mathcal{A}$ and an ideal $\mathcal{I}$ in it, we define $\mathcal{Q}$ as the category with
\begin{itemize}
\item
$\Ob(\mathcal{Q}) = \Ob(\mathcal{A})$,
\item
$\mathcal{Q}(x, y) = \mathcal{A}(x,y) / \mathcal{I}(x, y)$ for all objects $x,y$.
\end{itemize}
The quotient maps $\mathcal{A}(x,y)\to\mathcal{Q}(x,y)$ on morphisms provide a ringoid map $\mathcal{A}\to\mathcal{Q}$.
\end{definition}

In ring theory, we have
\begin{itemize}
\item
any intersection of subrings of ring $R$ is again a subring of $R$,
\item
any intersection of ideals of ring $R$ is again an ideal of $R$,
\item
the intersection of an ideal and a subring is an ideal of the subring.
\end{itemize}
The next lemma is their counterparts in the ringoid theory.
Since the proof is straightforward, we will omit it.

\begin{lemma}[Intersection]
Let $\mathcal{A}$ be a ringoid. Then we have
\begin{itemize}
\item
any intersection of subringoids of $\mathcal{A}$ is again a subringoid of $\mathcal{A}$,
\item
any intersection of ideals of $\mathcal{A}$ is again an ideal of $\mathcal{A}$,
\item
the intersection of an ideal $\mathcal{I}$ and a subringoid $\mathcal{B}$ is an ideal of the subringoid $\mathcal{B}$.
\end{itemize}
\end{lemma}


\begin{definition}[Subringoid generated by a set of morphisms]
Let $\mathcal{A}$ be a ringoid and $\mathcal{M}$ be a set of morphisms in it.
Then the subringoid of $\mathcal{A}$ generated by $\mathcal{M}$ is defined to be the smallest subringoid (intersection of all the ringoids) of $\mathcal{A}$ containing all the morphisms in $\mathcal{M}$.
More explicitly, the subringoid of $\mathcal{A}$ generated by $\mathcal{M}$ is the ringoid $\mathcal{B}$ such that $\Ob(\mathcal{B}) = \Ob(\mathcal{A})$, and for any two objects $x$ and $y$,
\[\mathcal{B}(x,y) = \left\{\sum_{i=1}^{n}a_{i}m_{i}\right\},\]
where $n\geq 0, a_{i}\in\mathbb{Z}$ and $m_{i}:x\to y$ is a composition of morphisms in $\mathcal{M}$.
Note that we allow the composition to be empty and $m_{i}$ to be identity morphism $\id:x\to x$.
It is easy to see that $\mathcal{B}$ is a subringoid of $\mathcal{A}$ and it is the smallest one containing $\mathcal{M}$.
\end{definition}

\begin{definition}[Ideal generated by a set of morphisms]
Let $\mathcal{A}$ be a ringoid and $\mathcal{M}$ be a set of morphisms in it.
Then the ideal of $\mathcal{A}$ generated by $\mathcal{M}$ is defined to be the smallest ideal (intersection of all the ideals) of $\mathcal{A}$ containing all the morphisms in $\mathcal{M}$.
More explicitly, the ideal of $\mathcal{A}$ generated by $\mathcal{M}$ is $\mathcal{I}$ with \[\mathcal{I}(x,y) = \left\{\sum_{i=1}^{n}a_{i}\cdot(f_{i}\circ m_{i}\circ g_{i})\right\},\]
where $n\geq 0, a_{i}\in\mathbb{Z}, g_{i}\in\mathcal{A}(x, x_{i}), m_{i}\in\mathcal{M}(x_{i}, y_{i}), f_{i}\in\mathcal{A}(y_{i}, y)$.
It is easy to see that $\mathcal{I}$ is an ideal of $\mathcal{A}$ and that it is the smallest one containing $\mathcal{M}$.
\end{definition}

\subsection{Adjunction between quivers and ringoids}
In this subsection, we will develop a pair of adjoint functors between the category of quivers and the the category of ringoids.
Then we will use this to give an alternative definition of subringoid generated by a set of morphisms.

\begin{definition}[Quiver]
A quiver $G$ consists of two sets $E$ and $V$ and two functions $s,t:E\rightrightarrows V$.
If $G = (E, V, s, t)$ and $G' = (E', V', s', t')$ are two quivers, a morphism $g:G\to G'$ is a pair of morphisms $g_{0}:V\to V'$ and $g_{1}:E\to E'$ such that $g_{0}\circ s = s'\circ g_{1}$ and $g_{0}\circ t = t'\circ g_{1}$.
Denote the category of quivers by $\mathbf{Quiver}$.
Intuitively, $E$ consists of oriented edges and $V$ consists of vertices.
\end{definition}

Denote the category of categories by $\mathbf{Cat}$.

\begin{definition}[Free functor from $\mathbf{Quiver}$ to $\mathbf{Cat}$]
We define the free functor $F:\mathbf{Quiver}\to\mathbf{Cat}$ to be the left adjoint to the forgetful functor $u:\mathbf{Cat}\to\mathbf{Quiver}$.
Below we give one construction of the free functor.
Given a quiver $G = (E, V)$, we construct a category $\mathcal{C} = FG$ such that
\begin{itemize}
\item
$\Ob(\mathcal{C}) = V$,
\item
the morphism set $\mathcal{C}(x,y)$ is the set of finite paths from $x$ to $y$ in the quiver $G$, where a path is defined as a finite sequence of composable edges and an ``empty path'' constitutes the identity morphisms of $\mathcal{C}$,
\item
the composition law of $\mathcal{C}$ follows from concatenation of paths in the quiver $G$.
\end{itemize}
\end{definition}

\begin{definition}[Free functor from $\mathbf{Cat}$ to $\mathbf{Ringoid}$]
We define the free functor $\mathbb{Z}:\mathbf{Cat}\to\mathbf{Ringoid}$ to be the left adjoint to the forgetful functor $u:\mathbf{Ringoid}\to\mathbf{Cat}$.
Here is one construction of the free functor $\mathbb{Z}$.
Given a category $\mathcal{C}$, we construct a ringoid $\mathcal{A} = \mathbb{Z}\mathcal{C}$ such that
\begin{itemize}
\item
$\Ob(\mathcal{A}) = \Ob(\mathcal{C})$,
\item
the morphism set $\mathcal{A}(x, y)$ is the free abelian group generated by $\mathcal{C}(x, y)$, where $x, y$ are any two objects in $\Ob(\mathcal{A}) = \Ob(\mathcal{C})$,
\item
the composition $\mathcal{A}(x,y)\times \mathcal{A}(y,z)\to \mathcal{A}(x, z)$ sends \[\left(\sum_{i=1}^{m}a_{i}f_{i}, \sum_{j=1}^{n}b_{j}g_{j}\right)\mapsto\sum_{i=1}^{m}\sum_{j=1}^{n}a_{i}b_{j}(g_{j}\circ f_{i}),\]
where $a_{i}, b_{j}$ are integers and $f_{i}\in\mathcal{C}(x,y), g_{j}\in\mathcal{C}(y, z)$.
\end{itemize}
\end{definition}

Composing those two pairs of adjoint functors above, we get a pair of adjoint functors between $\mathbf{Quiver}$ and $\mathbf{Ringoid}$.
Next, we will use that adjunction to give an alternative definition of the subringoid generated by a set of morphisms.

In general, the image of a functor is not necessarily a subcategory.
But as we will see in the following lemma, when the functor is ``nice'', the image will be a subcategory.

\begin{lemma}
If $F:\mathcal{A}\to\mathcal{B}$ is a functor  of categories such that $F:\Ob(\mathcal{A})\to\Ob(\mathcal{B})$ is bijective, then the image $\mathcal{C}:=F\mathcal{A}$ is a wide subcategory of $\mathcal{B}$.
\end{lemma}
\begin{proof}
We know $\Ob(\mathcal{C}) = \Ob(\mathcal{B}) = \Ob(\mathcal{A})$.
We also know that the morphism set $\mathcal{C}(x, y)$ is equal to the image of $F:\mathcal{A}(x, y)\to\mathcal{B}(x, y)$ for all objects $x$ and $y$.
So the identify morphisms exist in $\mathcal{C}$.

The composition law in $\mathcal{C}$ follows the composition law in $\mathcal{B}$.
We only need to check that if $f\in\mathcal{C}(x,y)$ and $g\in\mathcal{C}(y,z)$, then $g\circ f\in\mathcal{C}(x, z)$.
Since $f = F(f')$ for some $f'\in\mathcal{A}(x,y)$ and $g = F(g')$ for some $g'\in\mathcal{A}(y,z)$, we have $g\circ f = F(g')\circ F(f') = F(g'\circ f')$ with $g'\circ f'\in\mathcal{A}(x, z)$.
Therefore, $g\circ f\in\mathcal{C}(x, z)$.
\end{proof}

Furthermore, the following lemma will give us a condition on when the image of a morphism of ringoids is a subringoid.

\begin{lemma}
If $F:\mathcal{A}\to\mathcal{B}$ is a morphism of ringoids such that $F:\Ob(\mathcal{A})\to\Ob(\mathcal{B})$ is bijective, then the image $\mathcal{C}:=F\mathcal{A}$ is a subringoid of $\mathcal{B}$.
\end{lemma}
\begin{proof}
By the lemma above, $\mathcal{C}$ is a wide subcategory of $\mathcal{B}$.
Since $F:\mathcal{A}(x, y)\to\mathcal{B}(x, y)$ is a homomorphism of abelian groups for all objects $x$ and $y$, we know that $\mathcal{C}(x,y)$ is an abelian subgroup of $\mathcal{B}(x,y)$.
Therefore, by definition $\mathcal{C}$ is a subringoid of $\mathcal{B}$.
\end{proof}

\begin{definition}[Subringoid generated by a set of morphisms]
Let $\mathcal{A}$ be a ringoid and $\mathcal{M}$ be a set of morphisms in it.
Denote by $\mathcal{M}'$ the union of $\mathcal{M}$ and $\{\id:x\to x,\forall x\in\Ob(\mathcal{A})\}$.
Then we get a morphism of quivers $\mathcal{M}'\to u\mathcal{A}$ and by adjunction, this gives arise to a morphism of ringoids $F\mathcal{M}'\to\mathcal{A}$.
Since this morphism of ringoids is bijective on objects, its image is a subringoid of $\mathcal{A}$ by the lemma above.
We define this subringoid as the subringoid of $\mathcal{A}$ generated by morphisms in $\mathcal{M}$.
\end{definition}

\section{Unstable $A$-modules with the top $k$ squares}
\label{sec:U_k}

\subsection{Steenrod algebra and unstable modules over it}

\begin{definition}[Steenrod algebra]
The mod 2 Steenrod algebra $A$ is the quotient of the free unital graded $\mathbb{F}_{2}$-algebra generated by the elements $\Sq^{i}$ of degree $i$ by the ideal generated by
\[\Sq^{0} = 1, \Sq^{i} = 0\quad \textrm{when }i < 0\]
and the Adem relations
\begin{equation}
\label{eq:steenrod_relation}
\Sq^{i}\Sq^{j} = \sum_{t = 0}^{\lfloor i/2\rfloor}{j - t - 1 \choose i - 2t}\Sq^{i + j - t}\Sq^{t}\quad\textrm{when }0<i<2j.
\end{equation}
We shall denote by $\mathcal{M}$ the category of graded left $A$-modules, whose morphisms are $A$-linear maps of degree zero.
From now on, by an $A$-module we mean a module in the category $\mathcal{M}$.
Note that the Adem relations actually hold for all integers $i, j$.
\end{definition}

The following is a standard fact.
\begin{lemma}
The basis of the Steenrod algebra $A$ as a graded vector space over $\mathbb{F}_{2}$ is the admissible squares $\Sq^{i(1)}\Sq^{i(2)}\cdots\Sq^{i(m)}$ with $i(1)\geq 2i(2),i(2)\geq 2i(3),\dots,i(m-1)\geq 2i(m), i(m)> 0$.
Note that when $m = 0$, we have $\Sq^{0} = 1$.
\end{lemma}

\begin{notation}[Lower squares]
If $x$ is a homogeneous element in an $A$-module, then we denote \[\Sq_{i}x = \Sq^{|x| - i}x.\]
\end{notation}

\begin{proposition}[Adem relations in lower squares]
\label{prop:AdemLower}
Take any $A$-module $M$.
Let $i, j, n$ be any integers satisfying $n> j$ and $2n> i + j$.
Then for any homogeneous element $x$ of degree $n$ in $M$, we have
\[\Sq_{i}\Sq_{j}x = \sum_{s=\lceil (i+j)/2 \rceil}^{n}{s-j-1 \choose 2s-i-j}\Sq_{i+2j-2s}\Sq_{s}x.\]
\end{proposition}

\begin{proof}
Compute:
\[
\begin{split}
\Sq_{i}\Sq_{j}x
&= \Sq^{2n-i-j}\Sq^{n-j}x\\
&= \sum_{t=0}^{\lfloor n-(i+j)/2\rfloor}{n-j-t-1\choose 2n-i-j-2t}\Sq^{3n-i-2j-t}\Sq^{t}x\\
&= \sum_{t=0}^{\lfloor n-(i+j)/2\rfloor}{n-j-t-1\choose 2n-i-j-2t}\Sq_{-2n+i+2j+2t}\Sq_{n-t}x\\
&= \sum_{s=\lceil (i+j)/2 \rceil}^{n}{s-j-1 \choose 2s-i-j}\Sq_{i+2j-2s}\Sq_{s}x.
\end{split}
\]
The second equality comes from the Adem relations in upper squares.
The second to last equality comes from the substitution $s = n-t$.
\end{proof}

As in the case of upper squares, the Adem relations in lower squares hold for all integers $i,j,n$.

When $i>j$, all the terms of the summation on the right hand side satisfy $i+2j-2s\leq s$.
The reason is $s\geq (i+j)/2 \geq j$ and thus $(i+j-2s) + (j-s)\leq 0$.
So whenever there is a $\Sq_{i}\Sq_{j}$ with $i>j$, one can rewrite it as of a sum of $\Sq_{i'}\Sq_{j'}$'s such that $i'\leq j'$.
This observation agrees with Proposition \ref{prop:Gn} in a later section, which says whenever there is a sequence of lower squares, one can always rewrite it as a sum of $\Sq_{i(1)}\Sq_{i(2)}\cdots\Sq_{i(m)}$'s such that $i(1)\leq i(2)\leq\dots\leq i(m)$.

\begin{definition}[Unstable $A$-module]
An $A$-module is said to be unstable if $\Sq^{i}x = 0$ for any $i>n$ and any homogeneous element $x$ of degree $n$.
Or in other words, $M$ is unstable if $\Sq_{i}M = 0$ for all $i<0$.

Note that if $M$ is an unstable $A$-module, then $M^{n} = 0$ for all $n < 0$ because $x = \Sq^{0}x = 0$ for any homogeneous element $x$ in $M$ of degree $n < 0$.
We shall denote by $\mathcal{U}$ the full subcategory of $\mathcal{M}$ with objects the unstable ones.
From now on, by an unstable $A$-module we mean a module in the category $\mathcal{U}$.
\end{definition}

\begin{example}[Sphere module $S(n)$]
For any integer $n$, we define the sphere module $S(n)$ to be the $A$-module with the degree $n$ part equal to $\mathbb{F}_{2}$ and the other parts equal to zero.
The sphere module is defined for any integer $n$, but it is unstable only when $n \geq 0$.
\end{example}

\subsection{Intuition and formal definition}

Let $k$ be any natural number.
We will define a new kind of ``unstable $A$-module'', where the only simple Steenrod operations allowed are $\Sq_{i}$ with $i<k$.
As indicated by ``unstable'', we require $Sq_{i}=0$ when $i<0$ and thus $M^{n} = 0$ when $n < 0$.

For example, when $k=0$, the only Steenrod operations allowed are $\Sq_{i} = 0$ with $i <0$ and we have non-negatively graded $\mathbb{F}_{2}$-vector spaces.
When $k = 1$, we have one nontrivial Steenrod operation $\Sq_{0}$ which doubles the degree.
Note that $\Sq_{0}$ is equal to the identity on degree zero.

For another example, when $k = 3$, we have nontrivial Steenrod operations $\Sq_{2},\Sq_{1},\Sq_{0}$ and trivial Steenrod operations $\Sq_{i} = 0$ with $i<0$.
The Steenrod operations $\Sq_{2}$ and $\Sq_{1}$ are not available in all degrees --- $\Sq_{2}$ acts on degrees $\geq 2$ and $\Sq_{1}$ acts on degrees $\geq 1$.

We shall denote by $\mathcal{U}_{k}$ the category of such ``unstable $A$-modules''.
To make this idea clear, we are going to use the language of ringoids.
We are interested in
\begin{itemize}
\item
$\mathcal{M}$, the category of $A$-modules,
\item
$\mathcal{U}$, the category of unstable $A$-modules,
\item
$\mathcal{U}_{k}$, the category of ``unstable $A$-modules with only the top $k$ squares''.
\end{itemize}
Although $\mathcal{U}$ is not the category of all modules over any graded ring, it is the category of all modules over a certain ringoid.
In fact, we can formulate each of those three categories above as the category all modules over some ringoid.

\begin{definition}[Ringoid $\mathcal{R}$]
Let $\mathcal{R}$ be the ringoid such that
\begin{itemize}
\item
the objects are all the integers,
\item
for any $a, b\in\mathbb{Z}$, the morphism set $\mathcal{R}(a, b)$ is the $\mathbb{F}_{2}$-vector space whose basis are all finite sequences of integers $(c_{1},c_{2},\dots,c_{m})$ such that $a = c_{1} < c_{2} <\dots< c_{m} = b$.
\end{itemize}
We write the sequence \[(a=c_{1},c_{2},\dots,c_{m}=b)\] as \[\Sq^{c_{2}-c_{1}}\Sq^{c_{3}-c_{2}}\cdots\Sq^{c_{m}-c_{m-1}}.\]
For example, the morphism set $\mathcal{R}(-1, 2)$ is a four-dimentional $\mathbb{F}_{2}$-vector space with basis 
\[(-1,0,1,2), (-1,0,2),(-1,1,2),(-1,2)\]
or equivalently
\[\Sq^{1}\Sq^{1}\Sq^{1},\Sq^{1}\Sq^{2},\Sq^{2}\Sq^{1},\Sq^{3}.\]
The identity morphism $n\to n$ is also written as $\Sq^{0}$.
\end{definition}

\begin{definition}[Ringoid $\mathcal{A}$]
\label{def:ringoid_A}
Let $\mathcal{I}$ be the ideal of $\mathcal{R}$ generated by the Adem relations
\[\Sq^{i}\Sq^{j} - \sum_{t = 0}^{\lfloor i/2\rfloor}{j - t - 1 \choose i - 2t}\Sq^{i + j - t}\Sq^{t}\in\mathcal{R}(n, n + i + j)\textrm{ with }0<i<2j.\]
Define $\mathcal{A}$ as the quotient ringoid of $\mathcal{R}$ by the ideal $\mathcal{I}$.
This new ringoid $\mathcal{A}$ is the same as the ringoid such that
\begin{itemize}
\item
the objects are all the integers,
\item
for any $a, b\in\mathbb{Z}$, the morphism set $\mathcal{A}(a, b)$ is the degree $(b-a)$ part of the Steenrod algebra $A$.
\end{itemize}
The category of left modules over the ringoid $\mathcal{A}$ is exactly $\mathcal{M}$, the category of modules over the Steenrod algebra $A$.
\end{definition}

\begin{definition}[Ringoid $\mathcal{Q}$]
Let $\mathcal{J}$ be the ideal of $\mathcal{R}$ generated by the Adem relations (as in Definition \ref{def:ringoid_A}) and the unstability conditions
\[\Sq^{i}:n\to n + i\textrm{ with }i > n.\]
Define $\mathcal{Q}$ as the quotient ringoid of $\mathcal{R}$ by the ideal $\mathcal{J}$.
Then the category of left modules over the ringoid $\mathcal{Q}$ is exactly $\mathcal{U}$, the category of unstable modules over the Steenrod algebra $A$.
Note that $\mathcal{Q}$ can also be seen as the quotient of ringoid $\mathcal{A}$ by the ideal $\mathcal{L}$ generated by the unstability conditions alone.
\end{definition}

\begin{definition}[Ringoid $\mathcal{Q}_{k}$]
Let $k\geq 0$.
Let $\mathcal{Q}_{k}$ be the subringoid of $\mathcal{Q}$ generated by the top $k$ squares
\[\Sq^{i}:n\to n + i\textrm{ with }i\geq 0, n-i < k.\]
Denote the category of left modules over the ringoid $\mathcal{Q}_{k}$ by $\mathcal{U}_{k}$ and we call it the category of unstable modules over the Steenrod algebra $A$ with only the top $k$ squares.
\end{definition}

\begin{notation}
When $M$ is a module in $\mathcal{M},\mathcal{U}$ or $\mathcal{U}_{k}$, we use notations $M^{n}$ and $M(n)$ interchangeably throughout this paper to denote the degree $n$ part of module $M$.
\end{notation}

\begin{example}[Sphere module $S_{k}(n)$]
The sphere module $S_{k}(n)$ is defined to be the module in $\mathcal{U}_{k}$ with the degree $n$ part equal to $\mathbb{F}_{2}$ and the other parts equal to zero.
Note that $n$ cannot be negative because the negative degree parts of a module in $\mathcal{U}_{k}$ must be zero.
\end{example}

\subsection{Example: free modules}

\begin{definition}[Free modules]
By Definition \ref{def:moduleRingoid}, $\mathcal{A}(x,-)$ is an $\mathcal{A}$-module for any ringoid $\mathcal{A}$ and any object $x$ in the ringoid $\mathcal{A}$.
So when $\mathcal{A},\mathcal{Q},\mathcal{Q}_{k}$ are ringoids as defined in the last subsection,
\begin{itemize}
\item
$\mathcal{A}(n,-)$ is an $\mathcal{A}$-module for any integer $n$,
\item
$\mathcal{Q}(n,-)$ is a $\mathcal{Q}$-module for any integer $n$,
\item
$\mathcal{Q}_{k}(n,-)$ is a $\mathcal{Q}_{k}$-module for any integers $n$ and $k$ with $k\geq 0$.
\end{itemize}
These modules and their direct sums are said to be \emph{free} module in $\mathcal{M},\mathcal{U},\mathcal{U}_{k}$ respectively.
Denote $F(n):=\mathcal{Q}(n,-)$ and $F_{k}(n):=\mathcal{Q}_{k}(n,-)$.
Note that $F(n) = 0$ and $F_{k}(n) = 0$ for all $n<0$.
\end{definition}

From now on, we use $\iota_{n}$ to denote the universal element of degree $n$.

\begin{proposition}
\label{prop:Gn}
A basis of $\mathcal{A}(n,-)$ as a graded vector space over $\mathbb{F}_{2}$ is
\[\Sq_{i(1)}\Sq_{i(2)}\cdots\Sq_{i(m)}\iota_{n}\] with \[i(1)\leq i(2)\leq\dots\leq i(m)<n.\]
Note that when $m = 0$, we have $\iota_{n}$.
\end{proposition}
\begin{proof}
The admissible basis of the Steenrod algebra $A$ is \[\Sq^{j(1)}\Sq^{j(2)}\cdots\Sq^{j(m)}\] with \[j(s)>0, j(s)\geq 2j(s+1).\]
Therefore, a basis of $\mathcal{A}(n,-)$ is \[\Sq^{j(1)}\Sq^{j(2)}\cdots\Sq^{j(m)}\iota_{n}\] with \[j(s)>0, j(s)\geq 2j(s+1).\]
It translates into \[\Sq_{i(1)}\Sq_{i(2)}\cdots\Sq_{i(m)}\iota_{n}\] with \[i(1)\leq i(2)\leq\dots\leq i(m)<n.\qedhere\]
\end{proof}

We will write down the basis of $F(n)$ and $F_{k}(n)$ as vector spaces over $\mathbb{F}_{2}$ explicitly.
Before that, we present a lemma about the structure of $\mathcal{L}$, the ideal of ringoid $\mathcal{A}$ generated by $\Sq^{i}:n\to n+i$ with $i>n$.
Remember that the ringoid $\mathcal{Q}$ is equal to the quotient of $\mathcal{A}$ by $\mathcal{L}$.

\begin{lemma}
\label{lem:Lna_basis}
The $\mathcal{L}(n,a)$ is indeed a vector space generated by admissible Steenrod monomials \[\Sq^{j(1)}\Sq^{j(2)}\cdots\Sq^{j(m)}:n\to a,\]
where there exists at least one $s\in[1,m]$ such that \[j(s)>n+\sum_{t=s+1}^{m}j(t).\]
\end{lemma}
\begin{proof}
By definition, the $\mathcal{L}(n,a)$ as a vector space is generated by
\[\Sq^{j(1)}\Sq^{j(2)}\cdots\Sq^{j(m)}:n\to a,\]
where there exists at least one $s\in[1,m]$ such that \[j(s)>n+\sum_{t=s+1}^{m}j(t).\]
Using the Adem relations (\ref{eq:steenrod_relation}), we can rewrite the morphism above as a sum of some admissible Steenrod monomials.
Say we are applying the Adem relations on $\Sq^{i}\Sq^{j}:b\to b+i+j$ and get a sum of $\Sq^{i'}\Sq^{j'}:b\to c$.
If $i>j+b$, then $i'>j'+b$ because the Adem relations increases the first upper index and decreases the second one.
If $j>b$, then $i'>j'+b$ because $i' = i+j-j'$ and $i\geq 2j'$ together imply $i'\geq j'+j>j'+b$.
Therefore, for each admissible summand \[\Sq^{j'(1)}\Sq^{j'(2)}\cdots\Sq^{j'(m')}:n\to\ell\] on the right hand side of the Adem relations, there exists at least one $s\in[1,m']$ such that \[j'(s) > n + \sum_{t=s'+1}^{m'}j'(t).\]
So those admissible monomials generate the $\mathcal{L}(n,a)$.
\end{proof}

\begin{proposition}
\label{prop:Fn}
When $n\geq 0$, a basis of $F(n)$ as a graded vector space over $\mathbb{F}_{2}$ is
\[\Sq_{i(1)}\Sq_{i(2)}\cdots\Sq_{i(m)}\iota_{n}\] with \[0\leq i(1)\leq i(2)\leq\dots\leq i(m)< n.\]
Note that when $m = 0$, we have $\iota_{n}$.
\end{proposition}
\begin{proof}
We have $F(n) = \mathcal{Q}(n,-)$ and $F(n)^{a} = \mathcal{Q}(n,a) = \mathcal{L}(n,a)/\mathcal{I}(n,a)$.
The $\mathcal{Q}(n,-)$ is generated as a vector space over $\mathbb{F}_{2}$ by 
\[\Sq^{j(1)}\Sq^{j(2)}\cdots\Sq^{j(m)}\iota_{n}\]
with \[j(s)>0, j(s)\geq 2j(s+1), j(s)\leq n + \sum_{t = s+1}^{m}j(t).\]
They are linearly independent by Lemma \ref{lem:Lna_basis}.
It translates into \[\Sq_{i(1)}\Sq_{i(2)}\cdots\Sq_{i(m)}\iota_{n}\] with \[0\leq i(1)\leq i(2)\leq\dots\leq i(m)< n.\qedhere\]
\end{proof}

\begin{proposition}
\label{prop:Fkn}
When $n\geq 0$, a basis of $F_{k}(n)$ as a graded vector space over $\mathbb{F}_{2}$ is \[\Sq_{i(1)}\Sq_{i(2)}\cdots\Sq_{i(m)}\iota_{n}\] with \[0\leq i(1)\leq i(2)\leq\dots\leq i(m)< \min(n, k).\]
Note that when $m = 0$, we have $\iota_{n}$.
\end{proposition}
\begin{proof}
We have $F_{k}(n) = \mathcal{Q}_{k}(n,-)$ and $F_{k}(n)^{a} = \mathcal{Q}_{k}(n,a)$.
For all integers $n$ and $a$, $\mathcal{Q}_{k}(n,a)$ is an abelian subgroup of $\mathcal{Q}(n,a)$ generated by $\Sq_{0},\dots,\Sq_{k-1}$.
So we can compute a basis of $\mathcal{Q}_{k}(n,-)$ using a basis of $\mathcal{Q}(n,-)$ by selecting the ones using only $\Sq_{0},\dots,\Sq_{k-1}$.
Proposition \ref{prop:Fn} gives a basis of $\mathcal{Q}(n,-)$.
Therefore, a basis of $\mathcal{Q}_{k}(n,-)$ is \[\Sq_{i(1)}\Sq_{i(2)}\cdots\Sq_{i(m)}\iota_{n}\] with \[0\leq i(1)\leq i(2)\leq\dots\leq i(m)< \min(n, k).\qedhere\]
\end{proof}

\begin{remark}[Locally Noetherian]
The category $\mathcal{U}$ is locally noetherian (see \cite{schwartz1994} Chapter 1, Section 8).
But the categories $\mathcal{U}_{k}$ are not locally noetherian in general.
For example, $\mathcal{U}_{2}$ is not locally noetherian.
Although $F_{2}(2)$ is finitely generated, it has a submodule $M$ which is not finitely generated: Let $M$ be the submodule generated by $\Sq_{0}(\Sq_{1})^{i}\iota_{2}$ with $i\geq 0$.
\end{remark}

\subsection{Symmetric monoidal category}
We will construct a functor $\otimes:\mathcal{U}_{k}\times\mathcal{U}_{k}\to\mathcal{U}_{k}$, and then prove that $\mathcal{U}_{k}$ is a symmetric monoidal category with this tensor product and that $S_{k}(0)$ is the unit object with respect to this tensor product.
Before constructing the tensor product, we propose an alternative construction of ringoid $\mathcal{Q}_{k}$ in terms of generators and relations.

Recall that $\mathcal{Q} = \mathcal{R} / \mathcal{J}$ and $\mathcal{Q}_{k}$ is the subringoid of $\mathcal{Q}$ generated by the top $k$ squares.
In other words, $\mathcal{Q}_{k}$ is a subringoid of a quotient ringoid of $\mathcal{R}$.
We will present $\mathcal{Q}_{k}$ as a quotient ringoid of a subringoid of $\mathcal{R}$, after the following lemma.

\begin{lemma}
\label{lem:subquotient_ringoid}
Let $\mathcal{A}$ be a ringoid and $\mathcal{I}$ an ideal in it.
Let $\mathcal{M}$ be a set of morphisms in $\mathcal{A}$.
Then the subringoid $\mathcal{B}$ of $\mathcal{A}/\mathcal{I}$ generated by $\mathcal{M}$ is equivalent to the quotient ringoid of $\mathcal{C}$ by the ideal $\mathcal{I}$, where $\mathcal{C}$ is defined as the subringoid of $\mathcal{A}$ generated by $\mathcal{M}$ and $\mathcal{I}$.
\end{lemma}
\begin{proof}
The bijection on objects is easy to see.
We need to constract a bijection between $\mathcal{B}(x,y)$ and $(\mathcal{C}/\mathcal{I})(x,y)$ for any two objects $x,y$ of ringoid $\mathcal{A}$.
Let $\mathcal{D}$ be the subringoid of $\mathcal{A}$ generated by $\mathcal{M}$.
Then $\mathcal{B}(x,y) = (\mathcal{D}(x,y) + \mathcal{I}(x,y))/\mathcal{I}(x,y) = \mathcal{C}(x,y)/\mathcal{I}(x,y) = (\mathcal{C}/\mathcal{I})(x,y)$.
\end{proof}

\begin{proposition}
\label{prop:Q_k_gen_rel}
Let $k\geq 0$.
Let $\mathcal{R}_{k}$ be the subringoid of $\mathcal{R}$ generated by
\[\mathcal{M}:=\left\{\Sq^{i}:n\to n + i\textrm{ with }i\geq 0, 0\leq n-i<k\right\}.\]
Then $\mathcal{R}_{k} / (\mathcal{R}_{k}\cap \mathcal{J})$ is equivalent to the ringoid $\mathcal{Q}_{k}$.
\end{proposition}
\begin{proof}
According to Lemma \ref{lem:subquotient_ringoid}, $\mathcal{Q}_{k}$ is equivalent to the quotient ringoid of $\mathcal{C}$ by the ideal $\mathcal{J}$, where $\mathcal{C}$ is defined as the subringoid of $\mathcal{R}$ generated by $\mathcal{J}$ and $\Sq^{i}:n\to n+i$ with $i\geq 0, n-i<k$.
Note that if $i\geq 0$ and $n-i<0$, then $\Sq^{i}:n\to n+i$ is a morphism in ideal $\mathcal{J}$.
Therefore, $\mathcal{C}$ is the subringoid of $\mathcal{R}$ generated by $\mathcal{J}$ and $\mathcal{M}$.
Recall that $\mathcal{R}_{k}$ is the subringoid generated by $\mathcal{M}$.
Then the morphism set $\mathcal{C}(a, b) = \mathcal{J}(a, b) + \mathcal{R}_{k}(a,b)$ and \[\left(\frac{\mathcal{C}}{\mathcal{J}}\right)(a, b) = \frac{\mathcal{C}(a,b)}{\mathcal{J}(a,b)} = \frac{\mathcal{J}(a, b) + \mathcal{R}_{k}(a,b)}{\mathcal{J}(a,b)} = \frac{\mathcal{R}_{k}(a,b)}{\mathcal{J}(a,b)\cap\mathcal{R}_{k}(a,b)}.\]
Therefore, the ringoid $\mathcal{C}/\mathcal{J}$ is equivalent to $\mathcal{R}_{k} / (\mathcal{R}_{k}\cap \mathcal{J})$.
\end{proof}

After describing the ringoid $\mathcal{Q}_{k}$ in terms of generators and relations, we are now ready to define tensor product structure on the module category $\mathcal{U}_{k}$.

\begin{definition}[Tensor product]
Given two modules $M, N$ in $\mathcal{U}_{k}$, define a new module $M\otimes N$ in $\mathcal{U}_{k}$ as
\[(M\otimes N)(n) := \bigoplus_{i + j = n}M(i)\otimes N(j).\]
This is a finite direct sum because $M(i) = 0$ when $i < 0$ and similarly for $N$.
Let $x$ and $y$ be any nonzero homogeneous elements in $M$ and $N$ respectively.
We define the top $k$ Steenrod squares on $x\otimes y$ as
\[\Sq^{n}(x\otimes y) := \sum_{i + j = n,\,0\leq i\leq |x|,\, 0\leq j\leq |y|}\left(\Sq^{i}x\right)\otimes\left(\Sq^{j}y\right)\]
when $n\geq 0$ and $0\leq |x| + |y| - n < k$.
Note that on the right hand side, we only have the top $k$ squares because $|x| - i < k + n - |y| - i = k + j - |y|\leq k$ and similarly $|y| - j < k$.
This check uses the unstable condition $j\leq|y|$ and $i\leq|x|$.

We need to verify that such defined $M\otimes N$ is a module over $\mathcal{Q}_{k}$.
According to Proposition \ref{prop:Q_k_gen_rel}, it suffices to verify that $M\otimes N$ is a left $\mathcal{R}_{k}$-module with $f(x\otimes y) = 0$ for any morphism $f\in \mathcal{R}_{k}\cap\mathcal{J}$ and any two nonzero homogenous elements $x,y$ in $M,N$.
Since there is a similar tensor product structure on $\mathcal{U}$ and $\mathcal{Q} = \mathcal{R}/\mathcal{J}$, we know that $f(x\otimes y) = 0$ for any morphism $f\in\mathcal{J}$.
This finishes our proof that $M\otimes N$ is indeed a module over $\mathcal{Q}_{k}$.
\end{definition}

This tensor product $M\otimes N$ is functorial in both $M$ and $N$, so we get the tensor product functor $\otimes:\mathcal{U}_{k}\times\mathcal{U}_{k}\to\mathcal{U}_{k}$.
The sphere module $S_{k}(0)$ is the unit object with respect to this tensor product because $\left(M\otimes S_{k}(0)\right)(n) = M(n)$ and $\Sq^{n}(x\otimes y) = \left(\Sq^{n}x\right)\otimes y$, where $x$ is any homogenous element in $M$ and $y$ is the only nontrivial element in $S_{k}(0)$.
It is easy to further verify that $\left(\mathcal{U}_{k},\otimes, S_{k}(0)\right)$ is a symmetric monoidal category.

\section{Functors between categories $\mathcal{U}$ and $\mathcal{U}_{k}$}
\label{sec:functors}
\subsection{Forgetful functor}

\begin{definition}[Forgetful functor]
Inclusion morphisms of ringoids \[\mathcal{Q}_{0}\to\mathcal{Q}_{1}\to\dots\to\mathcal{Q}_{k-1}\to\mathcal{Q}_{k}\to\dots\to\mathcal{Q}\]
induce forgetful functors $u$ \[\mathcal{U}_{0}\leftarrow\mathcal{U}_{1}\leftarrow\dots\leftarrow\mathcal{U}_{k-1}\leftarrow\mathcal{U}_{k}\leftarrow\dots\leftarrow\mathcal{U}.\]
Those forgetful functors $u$ are additive.
\end{definition}

\begin{proposition}
The forgetful functors $u$ send free modules to free modules.
\end{proposition}
\begin{proof}
It suffices to prove that $u(F(n))$ and $u(F_{k+1}(n))$ are free modules in $\mathcal{U}_{k}$.
By Proposition \ref{prop:Fn}, $u(F(n))$ is a direct sum of $F_{k}(|x|)$ where $x = \Sq_{i(1)}\Sq_{i(2)}\cdots\Sq_{i(m)}\iota_{n}$ with $k\leq i(1)\leq\dots\leq i(m)<n$.
Similarly by Proposition \ref{prop:Fkn}, $u(F_{k+1}(n))$ is a direct sum of $F_{k}(|x|)$ where $x = (\Sq_{k})^{m}\iota_{n}$ with $m\geq 0$ if $k < n$ and $x = \iota_{n}$ if $k \geq n$.
\end{proof}

\begin{proposition}
\label{prop:prefor}
The forgetful functors $u$ send projective modules to projective modules.
\end{proposition}
\begin{proof}
Say the forgetful functor goes from category $\mathcal{U}$ to category $\mathcal{U}_{k}$.
Take any projective $M$ in $\mathcal{U}$.
By Proposition \ref{prop:projFree}, there is another module $N$ in $\mathcal{U}$ such that $M\oplus N$ is free in $\mathcal{U}$.
Applying the forgetful functor $u$, we get $u(M\oplus N)$ that is free in $\mathcal{U}_{k}$.
Since the forgetful functor $u$ is additive, we have $u(M)\oplus u(N)$ is free and therefore by Proposition \ref{prop:projFree} the module $u(M)$ is a projective in $\mathcal{U}_{k}$.
When the forgetful functor goes from $\mathcal{U}_{k+1}$ to $\mathcal{U}_{k}$, the proof is similar and omitted.
\end{proof}

\begin{proposition}
\label{prop:exfor}
The forgetful functor $u$ is exact.
\end{proposition}
\begin{proof}
Since the forgetful functor does not change the underlying set of a module and the underlying map between them, it must be exact.
\end{proof}

\subsection{Suspension functor}

\begin{definition}[Suspension morphism]
We define the suspension morphism $\sigma:\mathcal{A}\to\mathcal{A}$ of ringoids such that
\begin{itemize}
\item
$\sigma(n) = n - 1$,
\item
$\sigma\left(\Sq^{i}\right) = \Sq^{i}$.
\end{itemize}
The suspension morphism $\sigma:\mathcal{A}\to\mathcal{A}$ induces a suspension morphism of the quotient ringoids $\sigma':\mathcal{Q}\to\mathcal{Q}$ because $\sigma(\mathcal{I}(n,n+i))\subseteq\mathcal{I}(n-i,n+i-1)$.
Furthermore, the suspension morphism $\sigma':\mathcal{Q}\to\mathcal{Q}$ induces another suspension morphism of ringoids $\sigma_{k}:\mathcal{Q}_{k+1}\to\mathcal{Q}_{k}$ because $\sigma$ sends $\Sq_{i} = \Sq^{n-i}\in\mathcal{A}(n,2n-i)$ with $i < k + 1$ to $\Sq_{i-1} = \Sq^{n-i}\in\mathcal{A}(n-1,2n-i-1)$ with $i - 1 < k$.
\end{definition}

\begin{definition}[Suspension functor]
The suspension morphism $\sigma:\mathcal{Q}\to\mathcal{Q}$ of ringoids induces a suspension functor $\Sigma:\mathcal{U}\to\mathcal{U}$.
Similarly, the suspension morphism $\sigma_{k}:\mathcal{Q}_{k+1}\to\mathcal{Q}_{k}$ of ringoids induces a suspension functor $\Sigma:\mathcal{U}_{k}\to\mathcal{U}_{k+1}$.

The suspension functors are exact, because they just shift the underlying sets and maps.
Let $M$ be any module in $\mathcal{U}_{k}$.
Then the underlying sets and the Steenrod operations of $\Sigma M$ are
\[(\Sigma M)^{n + 1} = M^{n}\quad\forall n\]
and
\[\Sq_{i + 1}(\Sigma x) = \Sigma(\Sq_{i}x)\quad\forall i < k,\]
where $\Sigma x$ denotes the element in $(\Sigma M)^{n+1}$ corresponding to a homogeneous $x$ in $M^{n}$.
In particular, we have $\Sq_{0}(\Sigma x) = \Sigma(\Sq_{-1}x) = 0$.
\end{definition}

\begin{proposition}
\label{prop:invsus}
The suspension functor $\Sigma:\mathcal{U}_{k}\to\mathcal{U}_{k+1}$ restricts to an equivalence between $\mathcal{U}_{k}$ and the full subcategory of $\mathcal{U}_{k+1}$ with objects the ones with $\Sq_{0} = 0$.
\end{proposition}
\begin{proof}
Denote the full subcategory by $\mathcal{C}$.
It suffices to find a functor $F:\mathcal{C}\to\mathcal{U}_{k}$ such that $\Sigma F = 1$ and $F\Sigma  = 1$.
Here is the construction of the functor $F$.
Given a module $M$ in $\mathcal{C}$, we observe that $M^{0} = 0$ because if $x$ is a nonzero homogeneous element of degree zero in $M$, then $x = \Sq_{0}x = 0$.
Construct $FM$ as the one-degree downward shift of $M$.
More precisely, let \[(FM)^{n} = M^{n+1}\quad\forall n\geq 0\] and
\[\begin{tikzcd}
(FM)^{n} \arrow[r, "\Sq^{i}"]\arrow[d, "="]& (FM)^{n+i}\arrow[d, "="]\\
M^{n+1} \arrow[r, "\Sq^{i}"] & M^{n + i + 1}
\end{tikzcd}\quad\forall 0\leq i\leq n, i > n - k.\]
It is easy to check that $FM$ is indeed a module in $\mathcal{U}_{k}$ and $F$ is a functor from $\mathcal{C}$ to $\mathcal{U}_{k}$.
It is also easy to check that $\Sigma F = 1$ and $F\Sigma = 1$.
\end{proof}

\subsection{Frobenius functor and loop functor}
We are going to introduce a functor $\Phi:\mathcal{U}_{k}\to\mathcal{U}_{2k}$, which is a analogue of the Frobenius functor $\Phi:\mathcal{U}\to\mathcal{U}$ described in Section 1.7 of \cite{schwartz1994}.
To do that, we need an alternative story of the ringoids $\mathcal{A},\mathcal{Q},\mathcal{Q}_{k}$.

\begin{definition}[Ringoid $\mathcal{A}^{+}$]
Let $\mathcal{A}^{+}$ be the ringoid with objects $\{+\}\cup\mathbb{Z}$.
The morphism set $\mathcal{A}^{+}(n, n+i)$ is defined to be the degree $i$ part of the Steenrod algebra.
The morphism sets $\mathcal{A}^{+}(+,+), \mathcal{A}^{+}(n,+), \mathcal{A}^{+}(+, n)$ are all zero.
If $M$ is a left $\mathcal{A}^{+}$-module, i.e. a covariant additive functor from $\mathcal{A}^{+}$ to $\mathbf{Ab}$, then $M(+) = 0$ because $\mathcal{A}^{+}(+,+) = 0$.
Therefore, $\mathcal{A}^{+}\mathbf{Mod} = \mathcal{A}\mathbf{Mod} = \mathcal{M}$, where $\mathcal{M}$ denotes the category of modules over the Steenrod algebra $A$.
\end{definition}

\begin{definition}[Ringoid $\mathcal{Q}^{+}$]
Let $\mathcal{I}^{+}$ be the ideal of $\mathcal{A}^{+}$ generated by \[\Sq^{i}:n\to n+i\textrm{ with }i\geq 0, i>n.\]
Let $\mathcal{Q}^{+}$ be the quotient ringoid of $\mathcal{A}^{+}$ by the ideal $\mathcal{I}^{+}$.
Just as above, we have $\mathcal{Q}^{+}\mathbf{Mod} = \mathcal{Q}\mathbf{Mod} = \mathcal{U}$.
\end{definition}

\begin{definition}[Ringoid $\mathcal{Q}_{k}^{+}$]
Let $k\geq 0$.
Let $\mathcal{Q}_{k}^{+}$ be the subringoid of $\mathcal{Q}^{+}$ generated by
\[\Sq^{i}:n\to n + i\textrm{ with }i\geq 0, n-i<k.\]
Just as above, we have $\mathcal{Q}_{k}^{+}\mathbf{Mod} = \mathcal{Q}_{k}\mathbf{Mod} = \mathcal{U}_{k}$.
\end{definition}

The following lemma prepares us for the definition of the Frobenius morphism $\mathcal{Q}^{+}\to\mathcal{Q}^{+}$ of ringoids.
\begin{lemma}
\label{lem:froMor}
There is a unique morphism of ringoids $\phi:\mathcal{A}^{+}\to\mathcal{A}^{+}$ satisfying
\begin{itemize}
\item
$\phi(2n) = n, \phi(2n+1) = +, \phi(+) = +$,
\item
$\phi\left(\Sq^{2i}\right) = \Sq^{i}$.
\end{itemize}
\end{lemma}
\begin{proof}
It suffices to prove that the Frobenius morphism sends
\begin{equation}
\label{eq:froMor}
\Sq^{i}\Sq^{j} - \sum_{t = 0}^{\lfloor i/2\rfloor}{j - t - 1 \choose i - 2t}\Sq^{i + j - t}\Sq^{t}
\end{equation}
to zero.
If $i+j$ is odd, every summing term is sent to zero.
So we assume that $i+j$ is even.
If both $i$ and $j$ are odd, then we need to prove that \[{j-t-1\choose i-2t}\equiv 0\mod 2\quad\textrm{ when }t\textrm{ is even}.\]
It is true because of Lucas's theorem.
From now on, we assume that both $i$ and $j$ are even.
The Frobenius morphism sends (\ref{eq:froMor}) to \[\Sq^{i/2}\Sq^{j/2} - \sum_{s = 0}^{\lfloor i/4\rfloor}{j - 2s - 1 \choose i - 4s}\Sq^{i/2 + j/2 - s}\Sq^{s}.\]
We know that \[0 = \Sq^{i/2}\Sq^{j/2} - \sum_{s=0}^{\lfloor i/4\rfloor}{j/2 - s - 1 \choose i/2 - 2s}\Sq^{i/2 + j/2 - s}\Sq^{s}.\]
So it suffices to prove \[{j - 2s - 1 \choose i - 4s} \equiv {j/2 - s - 1 \choose i/2 - 2s}\mod 2,\]
or equivalently \[{2a+1 \choose 2b} \equiv {a\choose b}\mod 2.\]
It is true again by Lucas's theorem.
\end{proof}

\begin{definition}[Frobenius morphisms]
We define the Frobenius morphism $\phi:\mathcal{A}^{+}\to\mathcal{A}^{+}$ of ringoids such that
\begin{itemize}
\item
$\phi(2n) = n, \phi(2n+1) = +, \phi(+) = +$,
\item
$\phi\left(\Sq^{2i}\right) = \Sq^{i}$.
\end{itemize}
As we have seen in Lemma \ref{lem:froMor}, there is a unique morphism $\phi$ satisfying these properties.
The Frobenius morphism $\mathcal{A}^{+}\to\mathcal{A}^{+}$ induces a Frobenius morphism of the quotient ringoids $\mathcal{Q}^{+}\to\mathcal{Q}^{+}$ because $\phi$ sends $\Sq^{i}\in\mathcal{A}^{+}(n, n+i)$ with $i<n$ to $\Sq^{i/2}\in\mathcal{A}^{+}(n/2, n/2+i/2)$ with $i/2<n/2$ if both $n$ and $i$ are even, and zero otherwise.
Furthermore, the Frobenius morphism $\mathcal{Q}^{+}\to\mathcal{Q}^{+}$ induces another Frobenius morphism of ringoids $\mathcal{Q}_{2k}^{+}\to\mathcal{Q}_{k}^{+}$ because $\phi$ sends $\Sq_{i} = \Sq^{n-i}\in\mathcal{A}^{+}(n,2n-i)$ with $i < 2k$ to $\Sq_{i/2} = \Sq^{n/2-i/2}\in\mathcal{A}^{+}(n/2,n-i/2)$ with $i/2<k$ if both $n$ and $i$ are even, and zero otherwise.
\end{definition}

\begin{definition}[Frobenius functors]
The Frobenius morphism $\mathcal{Q}^{+}\to\mathcal{Q}^{+}$ of ringoids induces a Frobenius functor $\Phi:\mathcal{U}\to\mathcal{U}$.
Similarly, the Frobenius morphism $\mathcal{Q}^{+}_{2k}\to\mathcal{Q}^{+}_{k}$ of ringoids induces a Frobenius functor $\Phi:\mathcal{U}_{k}\to\mathcal{U}_{2k}$.
\end{definition}

\begin{remark}
If $M$ is a module in $\mathcal{U}_{k}$, then $\Phi M$ is a module in $\mathcal{U}_{2k}$ with \[(\Phi M)^{2n} = M^{n},\quad(\Phi M)^{\text{odd}} = 0.\]
We denote the element in $(\Phi M)^{2n}$ corresponding to $x$ in $M^{n}$ by $\Phi x$.
The Steenrod operations on $\Phi M$ are \[\Sq_{2i}(\Phi x) = \Phi(\Sq_{i}x),\quad\Sq_{\text{odd}}(\Phi x) = 0.\]
\end{remark}

\begin{proposition}
\label{prop:exfro}
The Frobenius functor is exact.
\end{proposition}
\begin{proof}
The Frobenius functor only shifts the underlying sets and maps.
\end{proof}

For any $k>0$, we have a natural transformation $\phi u\to\id$ between morphisms of ringoids $\mathcal{Q}_{k}^{+}\to\mathcal{Q}_{k}^{+}$, where $u$ is the forgetful morphism $\mathcal{Q}_{k}^{+}\to\mathcal{Q}_{2k}^{+}$ and $\phi$ is the Frobenius morphism $\mathcal{Q}_{2k}^{+}\to\mathcal{Q}_{k}^{+}$.
The natural transformation is given by
\begin{equation*}\begin{split}
&\Sq_{0}:\phi u(2n) = n\to 2n,\\
&0:\phi u(2n+1)=+\to 2n+1,\\
&0:\phi u(+) = +\to +.
\end{split}\end{equation*}
This natural transformation gives rise to another natural transformation $\lambda:u\Phi\to\id$ between functors $\mathcal{Q}_{k}^{+}\Mod\to\mathcal{Q}_{k}^{+}\Mod$, i.e. between functors $\mathcal{U}_{k}\to\mathcal{U}_{k}$.
The map $\lambda_{M}:u\Phi M\to M$ of modules in $\mathcal{U}_{k}$ sends $\Phi x\mapsto \Sq_{0}x$.
The kernel and cokernel of $\lambda_{M}:u\Phi M\to M$ are suspensions because $\Sq_{0}$ acts trivially on both of them.
We define functors $\Omega,\Omega_{1}:\mathcal{U}_{k}\to\mathcal{U}_{k - 1}$ such that $\Sigma\Omega M$ is the the cokernel of $\lambda_{M}$ and $\Sigma\Omega_{1} M$ is the kernel of $\lambda_{M}$.
So we have an exact sequence in $\mathcal{U}_{k}$ \[0\to\Sigma\Omega_{1}M\to u\Phi M\to M\to\Sigma\Omega M\to 0.\]

\begin{proposition}
\label{prop:omeSig}
Let $k$ be any positive integer.
The loop functor $\Omega:\mathcal{U}_{k}\to\mathcal{U}_{k-1}$ is the left adjoint of the suspension functor $\Sigma:\mathcal{U}_{k-1}\to\mathcal{U}_{k}$.
The functor $\Omega_{1}:\mathcal{U}_{k}\to\mathcal{U}_{k-1}$ is the first left derived functor of $\Omega$, and all higher derived functors are trivial.
\end{proposition}
\begin{proof}
We will first prove that $\Omega$ is the left adjoint to $\Sigma$.
Given any morphism $M\to\Sigma N$, the composition $u\Phi M\to M\to \Sigma N$ is equal to zero because $\Sq_{0}(\Sigma x) = 0$.
So by the universal property of cokernel, we get $\Sigma\Omega M\to \Sigma N$ and thus get $\Omega M\to N$ by Proposition \ref{prop:invsus}.
Similarly, given any morphism $\Omega M\to N$, we get $\Sigma\Omega M\to \Sigma N$ and composing it with $M\to \Sigma \Omega M$, we get $M\to\Sigma N$.
Thus we get a bijection between the two morphism sets $\mathcal{U}_{k}(\Sigma M, N)$ and $\mathcal{U}_{k-1}(M,\Omega N)$.

Now let us prove $\Omega_{1}$ is the first left derived functor of $\Omega$ and all higher derived functors are trivial.
Let $M$ be any module in $\mathcal{U}_{k}$.
For now, let us denote the kernel of $u\Phi M\to M$ by $\Sigma FM$.
So we are going to prove $F = \Omega_{1}$ and $\Omega_{i} = 0$ for $i \geq 2$.
Consider a free resolution of $M$ in the category $\mathcal{U}_{k}$
\[\dots\to P_{1}\to P_{0}\to M\to 0.\]
Then by definition, $\Omega_{i}(M)$ is equal to the $i$-th homology of \[\dots\to \Omega(P_{2})\to \Omega(P_{1})\to \Omega(P_{0})\to 0.\]
Think about the double complex $K^{*,*}$
\[\begin{tikzcd}
\dots \arrow[r]\arrow[d] &u\Phi P_{2} \arrow[r]\arrow[d] &u\Phi P_{1} \arrow[d] \arrow[r] & u\Phi P_{0} \arrow[r]\arrow[d]& 0\\
\dots \arrow[r] &P_{2}\arrow[r] &P_{1} \arrow[r] & P_{0} \arrow[r] & 0.
\end{tikzcd}\]
We get two spectral sequences from the double complex $K^{*,*}$ above.
Since for every $n\in\mathbb{Z}$, there are only finitely many nonzero $K^{p,q}$ with $p+q = n$, the two spectral sequence converge to the same thing.
Now we are going to compute the limits of those two spectral sequences.

Since both the Frobenius functor $\Phi$ and the forgetful functor $u$ are exact by Proposition \ref{prop:exfro} and \ref{prop:exfor}, we get $u\Phi M\to M$ by taking horizonal homology first.
Then take vertical homology, we get $\Sigma FM$ and $\Sigma\Omega M$.
If instead taking vertical homology first, we get \[\dots, 0, 0, 0\quad\mathrm{and}\quad\dots, \Sigma\Omega P_{2}, \Sigma\Omega P_{1}, \Sigma\Omega P_{0}.\]
Then taking horizontal homology, we get $\dots,\Sigma\Omega_{2}M, \Sigma\Omega_{1}M, \Sigma\Omega M$ because the suspension functors $\Sigma$ are exact.

Since the limits of those two spectral sequences are expected to be the same, we get $F = \Omega_{1}$ and $\Omega_{i} = 0$ for all $i\geq 2$.
\end{proof}

\begin{proposition}
\label{prop:omePro}
The loop functor $\Omega:\mathcal{U}_{k+1}\to\mathcal{U}_{k}$ preserves projectives.
\end{proposition}
\begin{proof}
It suffices to prove that its right adjoint $\Sigma:\mathcal{U}_{k}\to\mathcal{U}_{k+1}$ preserves epimorphisms.
Remember that the $\Sigma$ suspension functor is exact.
\end{proof}

\begin{proposition}
When $M$ is a sphere module in $\mathcal{U}_{k}$ and $k>0$,
\[
\Omega_{1}(S_{k}(n)) = \left\{
\begin{split}
&S_{k-1}(2n-1)&\textrm{ if }n > 0\\
&0&\textrm{ if }n = 0
\end{split}\right.
\]
\[
\Omega(S_{k}(n)) = \left\{
\begin{split}
&S_{k-1}(n-1)&\textrm{ if }n > 0\\
&0&\textrm{ if }n = 0
\end{split}\right.
\]
\end{proposition}
\begin{proof}
We have $u\Phi S_{k}(n) = S_{k}(2n)$ and the map $u\Phi S_{k}(n)\to S_{k}(n)$ sends $\iota_{2n}$ to $\Sq_{0}\iota_{n}$.
When $n > 0$, we have $\Sq_{0}\iota_{n} = 0$ and thus \[\Sigma\Omega_{1}S_{k}(n) = S_{k}(2n),\quad \Sigma\Omega S_{k}(n) = S_{k}(n).\]
When $n = 0$, we have $\Sq_{0}\iota_{n} = \iota_{n}$ and thus \[\Sigma\Omega_{1}S_{k}(n) = \Sigma\Omega S_{k}(n) = 0.\qedhere\]
\end{proof}

\section{Homological dimension of category $\mathcal{U}_{k}$}
\label{sec:U_k_homodim}
\begin{notation}
We abbreviate $\Ext_{\mathcal{U}_{k}}^{*}(M,N)$ to $\Ext_{k}^{*}(M,N)$.
\end{notation}

In this section, we shall prove that the homological dimension of the category $\mathcal{U}_{k}$ is at most $k$.
Our goal is to prove that $\Ext_{k}^{s}(M, N) = 0$ for all $s>k\geq 0$.
Our strategy is to first prove it for $N$ a sphere module by induction on $k$, then for $N$ bounded above, and finally for $N$ a general module.

\subsection{EHP sequence}
\begin{lemma}
\label{lem:U_k_EHP}
Let $M$ be any module in $\mathcal{U}_{k}$ and $N$ any module in $\mathcal{U}_{k-1}$.
Then we have the following long exact sequence of vector spaces over $\mathbb{F}_{2}$
\[\begin{tikzcd}
&&\dots\arrow[dll]\\
\Ext_{k-1}^{s}(\Omega M, N)\arrow[r] &\Ext_{k}^{s}(M, \Sigma N)\arrow[r] &\Ext_{k-1}^{s-1}(\Omega_{1}M, N)\arrow[dll]\\
\Ext_{k-1}^{s+1}(\Omega M, N)\arrow[r] &\Ext_{k}^{s+1}(M, \Sigma N)\arrow[r] &\Ext_{k-1}^{s}(\Omega_{1}M, N)\arrow[dll]\\
\dots&&
\end{tikzcd}\]
\end{lemma}
\begin{proof}
Recall from Proposition \ref{prop:omeSig} that the loop functor $\Omega:\mathcal{U}_{k}\to\mathcal{U}_{k-1}$ is the left adjoint of the suspension functor $\Sigma:\mathcal{U}_{k-1}\to\mathcal{U}_{k}$.
So \[\mathcal{U}_{k-1}(\Omega(-),N) = \mathcal{U}_{k}(-,\Sigma N)\] and its right derived functor is $\Ext_{k}^{*}(-,\Sigma N)$.
Since the inside functor $\Omega$ sends projectives to projectives by Proposition \ref{prop:omePro}, we have a Grothendieck spectral sequence with \[E_{2}^{s,t} = \Ext^{s}_{k-1}(\Omega_{t}M, N)\Rightarrow\Ext^{s+t}_{k}(M, \Sigma N).\]
Since $\Omega_{t} = 0$ for $t>1$, the $E_{2}$ page consists of only two nontrivial rows $t = 0, 1$.
We have the exact sequences 
\[\Ext^{s-2}_{k-1}(\Omega_{1} M, N)\to \Ext^{s}_{k-1}(\Omega M, N)\to E_{3}^{s, 0}\to 0\]
and
\[0\to E_{3}^{s-1, 1}\to \Ext^{s-1}_{k-1}(\Omega_{1}M, N)\to \Ext_{k-1}^{s+1}(\Omega M, N).\]
Since all further differentials are trivial, we have $E_{3} = E_{\infty}$.
By convergence of the spectral sequence, we have a short exact sequence
\[0\to E_{\infty}^{s, 0}\to \Ext_{k}^{s}(M,\Sigma N)\to E_{\infty}^{s-1, 1}\to 0.\]
Conbining these three exact sequences above, we get the long exact sequence.
\end{proof}

\begin{proposition}
\label{prop:hdsphere}
$\Ext_{k}^{s}(-, S_{k}(n)) = 0$ for all $s > k\geq 0$.
\end{proposition}
\begin{proof}
Say $N$ is the sphere module $S_{k}(n)$.
We will proceed by double induction on $n$ and $k$.
The base case is $n = 0$ or $k = 0$.
When $n = 0$, $\mathcal{U}_{k}(-,S_{k}(0))$ is an exact functor so $\Ext^{s}_{k}(-,S_{k}(0)) = 0$ for all $s > 0$.
When $k = 0$, $\mathcal{U}_{0}(-,S_{0}(n))$ is an exact functor so $\Ext^{s}_{0}(-,S_{0}(n)) = 0$ for all $s > k = 0$.

Now assume $n>0$ and $k>0$.
Our goal is to prove that \[\Ext_{k}^{s}(-, S_{k}(n)) = 0\] for all $s>k$.
Our induction hypothesis is that \[\Ext_{k'}^{s'}(-, S_{k'}(n')) = 0\] for all $n',k',s'$ satisfying $0\leq n'<n, 0\leq k'<k, s'>k'$.
Observe that $S_{k}(n) = \Sigma S_{k-1}(n-1)$.
Take $N = S_{k-1}(n-1)$ in the lemma above and we get a long exact sequence.
When $s>k$, we have $\Ext^{s}_{k-1}(\Omega M, N) = 0$ and $\Ext^{s-1}_{k-1}(\Omega_{1}M, N) = 0$ by the induction hypothesis.
Thus $\Ext^{s}_{k}(M, \Sigma N) = 0$.
\end{proof}

\subsection{Bounded above modules}
\begin{proposition}
\label{prop:hdbdd}
If $N$ is bounded above, i.e. $N^{n} = 0$ for large enough $n$, then $\Ext_{k}^{s}(-, N) = 0$ for all $s > k\geq 0$.
\end{proposition}
\begin{proof}
If $N = 0$, then it is trivial.
Let us assume $N\neq 0$.
Say the highest nontrivial degree of $N$ is equal to $n$.
We will proceed by induction on $n$.
The base case is $n = 0$.
When $n = 0$, $N$ is equal to a direct sum of sphere modules and therefore everything follows.
Let us assume $n > 0$ and that we have proven the cases $0,1,\dots, n-1$.
Then we have a short exact sequence of modules in $\mathcal{U}_{k}$
\[0\to N'\to N\to N''\to 0,\]
where $N'$ is the degree $n$ part of $N$ and $N''$ is the degree $<n$ part of $N$.
This short exact sequence induces a long exact sequence of $\Ext$ groups
\[\begin{tikzcd}
&&\dots\arrow[dll]\\
\Ext_{k}^{s}(M, N')\arrow[r] &\Ext_{k}^{s}(M, N)\arrow[r] &\Ext_{k}^{s}(M, N'')\arrow[dll]\\
\Ext_{k}^{s+1}(M, N')\arrow[r] &\Ext_{k}^{s+1}(M, N)\arrow[r] &\Ext_{k}^{s+1}(M, N'')\arrow[dll]\\
\dots&&
\end{tikzcd}\]
Since $N'$ is a direct sum of spheres, we that $\Ext_{k}^{s}(M, N') = 0$ by Proposition \ref{prop:hdsphere}.
We also know that $\Ext^{s}_{k}(M, N'') = 0$ by the induction hypothesis.
Therefore, $\Ext^{s}_{k}(M, N) = 0$.
\end{proof}

\subsection{General modules}
In this subsection, we are going to prove $\Ext_{k}^{s}(-,-) = 0$ for all $s>k\geq 0$.
Preparing for that proof, we present without proof the following lemma on the Milnor exact sequence for $\Ext$ groups.

\begin{lemma}[Milnor exact sequence]
Let $M$ be any module in $\mathcal{U}_{k}$.
Let $N_{0}\leftarrow N_{1}\leftarrow N_{2}\leftarrow \cdots$ be an inverse system of modules in $\mathcal{U}_{k}$ such that all maps are surjective.
Denote its inverse limit by $N$.
Then we have a short exact sequence \[0\to{\varprojlim}^{1}\Ext^{s-1}_{k}(M,N_{i})\to\Ext^{s}_{k}(M,N)\to\varprojlim\Ext^{s}_{k}(M,N_{i})\to 0.\]
\end{lemma}

\begin{theorem}
\label{thm:hdall}
$\Ext_{k}^{s}(-, -) = 0$ for all $s > k\geq 0$.
\end{theorem}
\begin{proof}
Let $M$ and $N$ be any two modules in $\mathcal{U}_{k}$.
We are going to prove $\Ext_{k}^{s}(M, N) = 0$ for all $s>k\geq 0$.
For any $i\geq 0$, define $N_{i}$ as a module in $\mathcal{U}_{k}$ with the degree $\leq i$ part equal to $N$ and the degree $> i$ part being zero.
That is, $N_{i}^{j} = N^{j}$ if $j \leq i$ and $N_{i}^{j} = 0$ if $j > i$.
The Steenrod opeartions on $N_{i}$ follow from $N$.
Then we get a surjective inverse system
\[N_{0}\leftarrow N_{1}\leftarrow N_{2}\leftarrow \cdots\]
and the inverse limit of the inverse system is exactly our module $N$.
So we get the Milnor exact sequence of $\Ext$ groups \[0\to{\varprojlim}^{1}\Ext^{s-1}_{k}(M,N_{i})\to\Ext^{s}_{k}(M,N)\to\varprojlim\Ext^{s}_{k}(M,N_{i})\to 0.\]
Since each $N_{i}$ is bounded above, we know $\Ext^{s}_{k}(M, N_{i}) = 0$ for all $i$ by Proposition \ref{prop:hdbdd} and thus the right term in the short exact sequence is zero.

It remains to prove that the left term in the short exact sequence is zero.
For all $i\geq 0$, we have a short exact sequence $0\to K\to N_{i+1}\to N_{i}\to 0$ with $K$ being a direct sum of sphere modules.
Therefore we have a long exact sequence, part of which looks like
\[\cdots\to\Ext_{k}^{s-1}(M,N_{i+1}) \to\Ext_{k}^{s-1}(M,N_{i}) \to\Ext_{k}^{s}(M,K) \to\cdots.\]
We know that $\Ext_{k}^{s}(M,K) = 0$ by Proposition \ref{prop:hdsphere}.
Therefore, the map $\Ext_{k}^{s-1}(M,N_{i+1})\to\Ext_{k}^{s-1}(M,N_{i})$ is a surjection for all $i$.
The left term is thus zero because it is $\varprojlim^{1}$ of a surjective inverse system.
\end{proof}

%
%

\begin{corollary}
Any module in $\mathcal{U}_{k}$ has a length $\leq k$ projective resolution.
In other words, the homological dimension of the category $\mathcal{U}_{k}$ is at most $k$.
\end{corollary}
\begin{proof}
By Proposition \ref{prop:enoughProj}, the abelian category $\mathcal{U}_{k}$ has enough projectives.
In an abelian category with enough projectives, if \[\Ext^{s}(M,-) = 0 \textrm{ for all }s>k\geq 0,\] then $M$ has a projective resolution of length $\leq k$.
So this corollary follows immediately from Theorem \ref{thm:hdall}.
\end{proof}

\section{$\Lambda$-complex for modules in $\mathcal{U}_{k}$}
\label{sec:Lambda_k}
In this section, we will introduce a contravariant functor $\Lambda_{k}$ from the category of unstable modules over the Steenrod algebra with only the top $k$ squares to the category of cochain complexes of graded vector spaces over $\mathbb{F}_{2}$, namely $\Lambda_{k}:\mathcal{U}_{k}^{\op}\to\Ch^{*}(\Gr(\mathbb{F}_{2}\Mod))$.
Note that we use the upper index to emphasize the \emph{cochain} complex.
The cohomological degree is denoted by $s$ and the degree in the graded vector space is denoted by $a$.
To motivate its study, we list two nice properties of this functor here:
\begin{itemize}
\item
The cohomology $H^{s,a}(\Lambda_{k}(M))$ is equal to $\Ext_{k}^{s}(M,S_{k}(a))$ for all $s,a$.
\item
The cochain complex $\Lambda_{k}(M)$ is relatively small and easy to compute.
To be more concrete, $\Lambda_{k}^{s}(M) = 0$ for all $s < 0$ or $s > k$.
Furthermore, when $M$ is finite, so is its cochain complex $\Lambda_{k}(M)$.
\end{itemize}

\subsection{Recall: $\Lambda$ algebra and $\Lambda$ functor}
In this subsection, we provide a brief recollection \cite{bousfield1966, priddy1970} of our knowledge on the $\Lambda$ algebra and the $\Lambda$ functor.

Formally, $\Lambda$ is an associative differential bigraded $\mathbb{F}_{2}$-algebra with generators $\lambda_{i}\in\Lambda^{1, i+1}$ for $i\geq 0$ and relations
\begin{equation}
\label{eq:lambda_relation}
\lambda_{i}\lambda_{2i+1+j} = \sum_{t\geq 0}{j-t-1\choose t}\lambda_{i+j-t}\lambda_{2i+1+t}\quad\textrm{for }i,j\geq 0
\end{equation}
with differential
\[d(\lambda_{i}) = \sum_{j\geq 1}{i-j\choose j}\lambda_{i-j}\lambda_{j-1}.\]
We refer to the first grading in $\Lambda$ as the cohomological degree $s$ and the second as the internal degree $t$.
The differential $d$ in $\Lambda$ increases $s$ by one and preserves $t$.

\begin{definition}[Admissible monomials]
A monomial \[\lambda_{I} := \lambda_{I(1)}\lambda_{I(2)}\cdots\lambda_{I(s)}\in\Lambda\] is said to be admissible if \[2I(r)\geq I(r+1)\quad\textrm{for all } 1\leq r < s.\]
The excess of $\lambda_{I}$ is defined as \[\textrm{excess}(I):=\sum_{r=1}^{s-1}(2I(r) - I(r+1)).\]
\end{definition}

\begin{proposition}
The admissible monomials form an additive basis for $\Lambda$.
\end{proposition}

\begin{proposition}
The internal degree $t$ part of the $s$-th cohomology of $\Lambda$ is equal to the $s$-th Ext group in the category of $\mathcal{M}$ from $S(0)$ to $S(t)$, i.e. \[H^{s,t}(\Lambda) = \Ext^{s}_{\mathcal{M}}(S(0),S(t)).\]
\end{proposition}

\begin{definition}[Subcomplex $\Lambda(m)$]
$\Lambda(m)$ is the defined to be the sub-bigraded vector space of $\Lambda$ spanned by the admissible monomials $\lambda_{I}$ with $I(1) < m$.
The trivial monomial $1$ lives in all $\Lambda(m)$.
\end{definition}

The $\Lambda(m)$ is a subcomplex of $\Lambda$ by virtue of Lemma \ref{lem:lambda_check}.

\begin{lemma}\label{lem:lambda_check}
\[d(\Lambda(m))\subseteq\Lambda(m),\quad\Lambda^{s,t}(m)\Lambda(m+t)\subseteq\Lambda(m).\]
\end{lemma}

\begin{proposition}
\label{prop:lambda_m_coho}
The internal degree $t$ part of the $s$-th cohomology of $\Lambda(m)$ is equal to the $s$-th Ext group in the category of $\mathcal{U}$ from $S(m)$ to $S(m+t)$, i.e. \[H^{s,t}(\Lambda(m)) = \Ext^{s}_{\mathcal{U}}(S(m),S(m+t)).\]
\end{proposition}

\begin{proposition}
\label{prop:lambda_EHP}
For each $m\geq 0$, there is a short exact sequence \[0\to\Lambda(m)\xrightarrow{e}\Lambda(m+1)\xrightarrow{h}\Sigma^{1,m+1}\Lambda(2m+1)\to 0,\]
where the suspension suspends $s$ and $t$ by 1 and $m+1$ respectively.
The map $e$ does not change the admissible monomials.
The map $h$ drops $\lambda_{n}$ if the admissible monomial starts with $\lambda_{n}$ and sends all other admissible basis vectors to zero.
\end{proposition}

This short exact sequence in Proposition \ref{prop:lambda_EHP} leads to a long exact sequence, known as the EHP sequence, for each $m$.

One can generalize $\Lambda(m)$ to $\Lambda(M)$ in such a way that the cohomology of $\Lambda(M)$ is equal to the $\Ext$ group from $M$ to spheres.

\begin{definition}
Given any $M\in\mathcal{U}$, we construct a cochain complex
\[\Lambda(M)=\bigoplus_{m}\Lambda(m)\otimes M_{m}\] with differential
\[d(\lambda_{I}\otimes x_{m}) = d(\lambda_{I})\otimes x_{m} + \sum_{i\geq 1, m-2i\geq 0}\lambda_{i-1}\lambda_{I}\otimes x_{m}\Sq^{i}.\]
Here $M_{m}$ denotes the dual $\Hom(M^{m}, \mathbb{F}_{2})$ of the degree $m$ part of $M$.
So the Steenrod operation $\Sq^{i}$ acts from the right on $x_{m}$.
We enforce the condition $m-2i\geq 0$ because otherwise $x_{m}\Sq^{i} = x_{m}\Sq_{2m-i} = 0$ by the unstability of the module $M$.
The $\Lambda(M)$, as a subspace of $\Lambda\otimes \left(M^{\vee}\right)$, closed under the differential by Lemma \ref{lem:lambda_check}.
The relations (\ref{eq:lambda_relation}) in the $\Lambda$ algebra and the ones (\ref{eq:steenrod_relation}) in the Steenrod algebra $A$ play so well with each other that one can check $d^{2}(\lambda_{I}\otimes x_{m}) = 0$.

The cochain complex $\Lambda(M)$ is still bigraded: the first grading is still the cohomological degree $s$, but the second grading is the \emph{absolute} internal degree $a$.
The absolute internal degree $a$ of $\lambda_{I}\otimes x_{m}$ is equal to $m$ plus the internal degree of $\lambda_{I}$.
The differential $d$ in $\Lambda(M)$ increases $s$ by one and preserves $a$.
\end{definition}

\begin{proposition}
\label{prop:lambda_coho}
The absolute internal degree $a$ part of the $s$-th cohomology of $\Lambda(M)$ is equal to the $s$-th Ext group in the category of $\mathcal{U}$ from $M$ to $S(a)$, i.e. \[H^{s,a}(\Lambda(M)) = \Ext^{s}_{\mathcal{U}}(M,S(a)).\]
\end{proposition}

We can view $\Lambda$ as a contravariant functor from $\mathcal{U}$ to $\Ch^{*}(\Gr(\mathbb{F}_{2}\Mod))$.

\subsection{Cochain complex $\Lambda_{k}(m)$}
\begin{definition}
For all $m, k\geq 0$, $\Gamma(m, k)$ is defined to be the sub-bigraded vector space of $\Lambda(m)$ spanned by the admissible monomials $\lambda_{I}$ with $I(1)< m$ and \[\textrm{excess}(I) + (s-1) > I(1) - (m-k).\]
The trivial monomial 1 does not live in any $\Gamma(m,k)$.
\end{definition}

The $\Gamma(m,k)$ is a subcomplex of $\Lambda(m)$ by virtue of the following lemma.

\begin{lemma}
$\Gamma(m,k)$ is closed under the differential.
\end{lemma}
\begin{proof}
Note that $\textrm{excess}(I) + (s-1) > I(1) - (m-k)$ if and only if $t+m-k-1 > 2I(s)$, where $t$ is the internal degree of $\lambda_{I}$.
Since the differential $d$ does not change $t, m, k$, it suffices to prove that the differential $d$ does not increase the last subscript.
In other words, we need to prove that $d(\lambda_{I})$ can be written as a sum of admissible monomials $\lambda_{J}$ with $J(s+1) \leq I(s)$.
This is true because both the differential formula and the relations in $\Lambda$ does not increase the second subscript.
\end{proof}

Observe that $\Gamma(m, k+1)$ is a subcomplex of $\Gamma(m, k)$.

\begin{definition}
For all $m,k\geq 0$, $\Lambda_{k}(m)$ is defined to be the quotient cochain complex of $\Lambda(m)$ by its subcomplex $\Gamma(m, k)$.
The differentials in $\Lambda_{k}(m)$ follow from those in $\Lambda(m)$.
Observe that all nontrivial admissible monomials $\lambda_{I}$ in $\Lambda_{k}(m)$ have $m-k\leq I(1)\leq m-1$, because $I(1)<m$ and $0\leq\textrm{excess}(I) + (s-1)\leq I(1) - (m-k)$.

Note that $\Lambda_{k}^{s}(m) = 0$ when $s > k$ or $s < 0$.
When $s > k$, we have $\textrm{excess}(I) + s - 1 < I(1) - m + s$ and thus $\textrm{excess}(I)\leq I(1) - m< 0$.
No admissible monomial $\lambda_{I}$ can have $\textrm{excess}(I) < 0$.
\end{definition}

In a later subsection, we are going to prove Theorem \ref{thm:lambda_k_coho}, a special case of which is that the cohomology $H^{s,t}(\Lambda_{k}(m))$ is equal to $\Ext_{k}^{s}(S_{k}(m),S_{k}(m+t))$ for all $s,t$.
This result is the correspondence of Proposition \ref{prop:lambda_m_coho} in the world $\mathcal{U}_{k}$.

\begin{proposition}
The short exact sequence in Proposition \ref{prop:lambda_EHP} induces a short exact sequence \[0\to\Lambda_{k}(m)\xrightarrow{e}\Lambda_{k+1}(m+1)\xrightarrow{h}\Sigma^{1,m+1}\Lambda_{k}(2m+1)\to 0,\]
where the suspension suspends $s$ and $t$ by 1 and $m+1$ respectively.
\end{proposition}
\begin{proof}
As a vector space over $\mathbb{F}_{2}$, the $\Lambda_{k}(m)$ is spanned by the admissible monomials $\lambda_{I}$ satisfying $I(1)<m$ and  \[\textrm{excess}(I) + (s-1) \leq I(1) - (m-k).\]
The map $e$ keeps the admissible monomials and the map $h$ drops $\lambda_{m}$ if the the admissible monomial starts with $\lambda_{m}$ and sends all other additive basis to zero.
It is easy to check the exactness of the short sequence and we omit it.
\end{proof}

The short exact sequence above leads to a long exact sequence
\[\begin{tikzcd}
&&\ldots\arrow[dll, "P"']\\
\Ext_{k}^{s}(m,n)\arrow[r, "E"'] &\Ext_{k+1}^{s}(m+1,n+1)\arrow[r, "H"] &\Ext_{k}^{s-1}(2m+1,n)\arrow[dll, "P"]\\
\Ext_{k}^{s+1}(m,n)\arrow[r, "E"'] &\Ext_{k+1}^{s+1}(m+1,n+1)\arrow[r, "H"] &\Ext_{k}^{s-1}(2m+1,n)\arrow[dll, "P"]\\
\ldots&&
\end{tikzcd}\]
where $\Ext_{k}^{s}(m,n)$ is an abbreviation for $\Ext_{k}^{s}(S_{k}(m),S_{k}(n))$.
Note that we've seen this long exact sequence before in Lemma \ref{lem:U_k_EHP}.
We see this long exact sequence as the correspondence of the EHP sequence in the world of $\mathcal{U}_{k}$.

\begin{example}
We write down the structure of several $\Lambda_{k}(m)$'s explicitly.
\begin{enumerate}
\item
$\Lambda_{0}(m)$ has additive basis $\{1\}$ for all $m\geq 0$.
All differentials are trivial.
\item
$\Lambda_{1}(m)$ has additive basis 
\[
\left\{
\begin{split}
&\{1\}&\textrm{ if }m = 0\\
&\{1,\lambda_{m-1}\}&\textrm{ if }m \geq 1
\end{split}\right.
\]
All differentials are trivial.
\item
$\Lambda_{2}(m)$ has additive basis
\[
\left\{
\begin{split}
&\{1\}&\textrm{ if }m = 0\\
&\{1,\lambda_{0},\lambda_{0}\lambda_{0}\}&\textrm{ if }m = 1\\
&\{1,\lambda_{m-2},\lambda_{m-1},\lambda_{m-1}\lambda_{2m-2}\}&\textrm{ if } m \geq 2
\end{split}\right.
\]
All differentials are trivial.
\item
$\Lambda_{3}(m)$ has additive basis
\[
\left\{
\begin{split}
&\{1\}&\textrm{ if }m = 0\\
&\{1,\lambda_{0},\lambda_{0}\lambda_{0},\lambda_{0}\lambda_{0}\lambda_{0}\}&\textrm{ if }m = 1\\
&\{1,\lambda_{0},\lambda_{1},\lambda_{0}\lambda_{0},\lambda_{1}\lambda_{1},\lambda_{1}\lambda_{2},\lambda_{1}\lambda_{2}\lambda_{4}\}&\textrm{ if }m = 2\\
&\{1,\lambda_{m-3},\lambda_{m-2},\lambda_{m-1},\lambda_{m-2}\lambda_{2m-4},\lambda_{m-1}\lambda_{2m-3},&\\
&\quad\lambda_{m-1}\lambda_{2m-2},\lambda_{m-1}\lambda_{2m-2}\lambda_{4m-4}\}&\textrm{ if } m \geq 3
\end{split}\right.
\]
All differentials are trivial.
\end{enumerate}
\end{example}

\subsection{Functor $\Lambda_{k}:\mathcal{U}_{k}^{\op}\to\Ch^{*}(\Gr(\mathbb{F}_{2}\Mod))$}

\begin{definition}
Given any $M\in\mathcal{U}_{k}$, we construct a cochain complex
\[\Lambda_{k}(M)=\bigoplus_{m}\Lambda_{k}(m)\otimes M_{m}\] with differentials
\[d(\lambda_{I}\otimes x_{m}) = d(\lambda_{I})\otimes x_{m} + \sum_{i\geq 1, 0\leq m-2i< k}\lambda_{i-1}\lambda_{I}\otimes x_{m}\Sq^{i}.\]
We require $0\leq m-2i<k$ because $x_{m}Sq^{i} = x_{m}\Sq_{2m-i}$ and as a module in $\mathcal{U}_{k}$, the $M$ only allows operations $\Sq_{0},\ldots,\Sq_{k-1}$.
Note that we can omit the condition $m-2i\geq 0$ in the definition of the differential because $x_{m}\Sq^{i} = 0$ automatically when $m-2i < 0$.
\end{definition}

\begin{lemma}
The differential $d$ is well-defined.
\end{lemma}
\begin{proof}
It suffices to prove $\lambda_{i-1}\lambda_{I}\in\Gamma(m-i,k)$ if $\lambda_{I}\in\Gamma(m, k)$ when $k\geq 0, i\geq 1$ and $m\geq 2i$.
By Lemma \ref{lem:lambda_check}, $2i\leq m$ implies that $\lambda_{i-1}\Lambda_{I}$ lives in $\Lambda(m-i)$.
Note that the condition for an admissible monomial $\lambda_{I}\in\Lambda(m)$ to be in $\Gamma(m,k)$ is $\textrm{excess}(I)  + (s-1) > I(1) - (m-k)$, which is equivalent to $t(I) - 2I(s) > k - m + 1$.
Here $t(I)$ denotes the internal degree of $\lambda_{I}$.
Say $\lambda_{i-1}\lambda_{I} = \sum_{J}\lambda_{J}$ with $\lambda_{J}$ being admissible monomials.
Then $t(J) = i + t(I)$ and $t(J) - 2J(s+1) = i + t(I) - 2J(s+1) \ge i + t(I) - 2I(s) > i + k - m + 1$.
So $\lambda_{J}\in\Gamma(m-i, k)$.
\end{proof}

\begin{theorem}
\label{thm:lambda_k_dsq}
$d^{2} = 0$ in $\Lambda_{k}(M)$.
\end{theorem}

We first prove two lemmas which will come in handy when proving Theorem \ref{thm:lambda_k_dsq}.
Lemma \ref{lem:lambda_comp_0} is about the cochain complex $\Lambda_{k}(m)$ and Lemma \ref{lem:lambda_ineq} is about the standard $\Lambda$ algebra.

\begin{lemma}
\label{lem:lambda_comp_0}
If $i + 1 \leq m - k$ and $\lambda_{I}$ is any admissible monomial in $\Lambda(m+i+1)$, then $\lambda_{i}\lambda_{I} = 0 \in \Lambda_{k}(m)$.
\end{lemma}
\begin{proof}
By Lemma \ref{lem:lambda_check}, $\lambda_{i}\lambda_{I}$ lives in $\Lambda(m)$.
If $\lambda_{I} = 1$, then $\lambda_{i}$ is trivial in $\Lambda_{k}(m)$ because $i<m-k$.
From now on, assume the length of $\lambda_{I}$ is $s\geq 1$.
Say $\lambda_{i}\lambda_{I} = \sum_{J}\lambda_{J}$ where each $\lambda_{J}$ is admissible of length $s + 1$ in $\Lambda(m)$.
It suffices to prove $\textrm{excess}(J) + s > J(1) - (m-k)$ or equivalently, $J(1) + \cdots + J(s) + s + (m-k) > J(s+1)$.
Since $i + 1\leq m - k$, it suffices to prove $J(1) + \cdots + J(s) + i + (s+1) > J(s+1)$, which follows from Lemma \ref{lem:lambda_ineq}
\end{proof}

\begin{lemma}
\label{lem:lambda_ineq}
Let $\lambda_{I}$ be any admissible monomial of length $s\geq 1$.
Write $\lambda_{i}\lambda_{I}$ as the sum of admissible monomials $\lambda_{J}$'s.
Then $i + J(1) + \cdots + J(s)\geq J(s+1)$.
\end{lemma}
\begin{proof}
We will proceed by induction on $s$.
The base case is $s = 1$.
When $s = 1$, write $\lambda_{i}\lambda_{j}$ as the sum of admissible monomials $\lambda_{i'}\lambda_{j'}$'s.
If $2i\geq j$, then $i' = i$ and $j' = j$.
Otherwise, in the relation (\ref{eq:lambda_relation}), we have $i + (i + j -t ) \geq 2i + 1 + t$ because the binomial coefficient requires $j - t - 1\geq t$.
One can check that the right hand side of the relation (\ref{eq:lambda_relation}) is admissible.

Assume $s > 1$.
Write $\lambda_{i}\lambda_{I(1,2,\ldots,s-1)}$ as the sum of admissible monomials $\lambda_{P}$.
Then by the case $s-1$, we have $i + P(1) + \cdots P(s-1)\geq P(s)$.
Now consider $\lambda_{P}\lambda_{I(s)} = \lambda_{P(1,\ldots,s-1)}\lambda_{P(s)}\lambda_{I(s)}$.
It is equal to a sum of $\lambda_{P(1,\ldots,s-1)}\lambda_{Q(1,2)}$ where both $\lambda_{P(1,\ldots,s-1)}$ and $\lambda_{Q(1,2)}$ are admissible monomials.
By the base case $s = 1$, we have $P(s) + Q(1)\geq Q(2)$.
Adding those two inequalities leads to $i + P(1) + \cdots + P(s-1) + Q(1) \geq Q(2)$.
Further applying relations (\ref{eq:lambda_relation}) to $\lambda_{P(1,\ldots,s-1)}\lambda_{Q(1,2)}$ will either preserve both sides or increases left hand side while decreasing the right hand side.
\end{proof}

\begin{proof}[Proof to Theorem \ref{thm:lambda_k_dsq}]
In this proof, we declare ${a\choose b} = 0$ if $a < 0$ or $b < 0$.
Let $x$ be any element in $M_{m}$.
\[
\begin{split}
d^{2}\left(\lambda_{I}\otimes x\right)
=&d^{2}(\lambda_{I})\otimes x\\
&\quad+\sum_{n\geq 1, m-2n< k}\lambda_{n-1}d(\lambda_{I})\otimes x\Sq^{n}\\
&\quad+\sum_{n\geq 1, m-2n< k}d\left(\lambda_{n-1}\lambda_{I}\right)\otimes x\Sq^{n}\\
&\quad+\sum_{i,j\geq 1, m-2i< k, m-i-2j<k}\lambda_{j-1}\lambda_{i-1}\lambda_{I}\otimes x\Sq^{i}\Sq^{j}\\
=&\sum_{n\geq 2, m-2n< k}d(\lambda_{n-1})\lambda_{I}\otimes x\Sq^{n}\\
&\quad+\sum_{i,j\geq 1, m-2i< k, m-i-2j<k}\lambda_{j-1}\lambda_{i-1}\lambda_{I}\otimes x\Sq^{i}\Sq^{j}\\
\end{split}
\]
For convenience, define $S(m, k)$ to be the set of indices $(i,j)\in\mathbb{Z}_{>0}^{2}$ satisfying \[m-2i< k, m-i-2j<k.\]
Define $A, B, C$ as
\[
\begin{split}
A:=&\sum_{(i, j)\in S(m, k), i\geq 2j}\lambda_{j-1}\lambda_{i-1}\lambda_{I}\otimes x\Sq^{i}\Sq^{j},\\
B:=&\sum_{(i, j)\in S(m, k), i< 2j}\lambda_{j-1}\lambda_{i-1}\lambda_{I}\otimes x\Sq^{i}\Sq^{j},\\
C:=&\sum_{n\geq 2, m-2n< k}d(\lambda_{n-1})\lambda_{I}\otimes x\Sq^{n}.
\end{split}
\]
Therefore, $d^{2}(\lambda_{I}\otimes x) = A + B + C$.
In $A$, the $\Sq^{i}\Sq^{j}$ is admissible but $\lambda_{j-1}\lambda_{i-1}$ is not.
Applying the relations (\ref{eq:lambda_relation}), we get
\[
\begin{split}
A =& \sum_{(i,j)\in S(m,k), i\geq 2j}\sum_{t\ge 0}{i-2j-t-1\choose t}\lambda_{i-j-t-1}\lambda_{2j+t-1}\lambda_{I}\otimes x\Sq^{i}\Sq^{j}\\
=& \sum_{(i,j)\in S(m,k), i\geq 2j}\sum_{v\ge 2j}{i-v-1\choose v-2j}\lambda_{i+j-v-1}\lambda_{v-1}\lambda_{I}\otimes x\Sq^{i}\Sq^{j}\\
=& \sum_{(i', j', s, t)\in A(m, k)}{s-i'-1\choose i'-2t}\lambda_{j'-1}\lambda_{i'-1}\lambda_{I}\otimes x\Sq^{s}\Sq^{t},
\end{split}
\]
where $A(m, k)$ is defined as the set of indices $(i, j, s, t) \in \mathbb{Z}_{>0}^{4}$ satisfying \[i< 2j, s\geq 2t, i+j = s+t, m-s-2t < k.\]
In the second equality, we used the substitution $v = 2j + t$.
In the third equality, we used the substitution $i' = v, j' = i+j-v, s = i, t = j$.
It is straightforward to check the third equality and the corresondence of index sets.
In $B$, the $\lambda_{j-1}\lambda_{i-1}$ is admissible but $\Sq^{i}\Sq^{j}$ is not.
Applying the Adem relations, we get
\[
\begin{split}
B &= \sum_{(i,j)\in S(m, k), i < 2j}\sum_{t\geq 0}{j-t-1\choose i-2t}\lambda_{j-1}\lambda_{i-1}\lambda_{I}\otimes x\Sq^{i+j-t}\Sq^{t}\\
&= \sum_{(i,j)\in S(m, k), i < 2j}\sum_{t\geq 1}{j-t-1\choose i-2t}\lambda_{j-1}\lambda_{i-1}\lambda_{I}\otimes x\Sq^{i+j-t}\Sq^{t}\\
&\qquad+\sum_{(i,j)\in S(m, k), i < 2j}{j-1\choose i}\lambda_{j-1}\lambda_{i-1}\lambda_{I}\otimes x\Sq^{i+j}\\
&= D + E,
\end{split}
\]
where $D,E$ are defined as
\[
\begin{split}
D:=&\sum_{(i,j,s,t)\in D(m,k)}{s-i-1\choose i-2t}\lambda_{j-1}\lambda_{i-1}\lambda_{I}\otimes x\Sq^{s}\Sq^{t},\\
E:=&\sum_{(i,j)\in E(m,k)}{j-1\choose i}\lambda_{j-1}\lambda_{i-1}\lambda_{I}\otimes x\Sq^{i+j},
\end{split}
\]
$D(m,k)$ is defined to be the set of indices $(i,j,s,t)\in\mathbb{Z}_{>0}^{4}$ satisfying \[i<2j,s\ge 2t, i+j = s+t, m-2i < k\]
and $E(m,k)$ is defined to be the set of indices $(i,j)\in\mathbb{Z}_{>0}^{2}$ satisfying \[i<2j, m-2i < k.\]
Therefore, $d^{2}(\lambda_{I}\otimes x) = A + D + E + C$.
Applying the differential formula to $d(\lambda_{n-1})$, we get 
\[
\begin{split}
C =& \sum_{n\geq 2, m-2n< k}\sum_{i\geq 1}{n-i-1\choose i}\lambda_{n-i-1}\lambda_{i-1}\lambda_{I}\otimes x\Sq^{n}\\
=& \sum_{(i,j)\in C(m, k)}{j-1\choose i}\lambda_{j-1}\lambda_{i-1}\lambda_{I}\otimes x\Sq^{i+j},
\end{split}
\]
where $C(m, k)$ is defined as the set of indices $(i, j) \in \mathbb{Z}_{>0}^{2}$ satisfying \[i<2j, m-2(i+j)< k.\]
In the second equality, we used the substitution $j = n - i$.
Observe that $A$ and $D$ are summation of the same expression over slightly different index set.
The same for $E$ and $C$.

Take any $(i,j,s,t)\in D(m,k)$.
Then $(m-2i) - (m-s-2t) = s+2t-2i\geq 1$ because the binomial coefficient leads to $s-i-1\geq i-2t$.
So $(i,j,s,t)\in A(m,k)$ and $D(m,k)$ is a subset of $A(m,k)$.
The difference of those two index sets are consists of $(i,j,s,t)\in A(m,k)$ satisfying $m-s-2t = k-1$ and $s+2t-2i = 1$.
Any $(i,j,s,t)$ in the difference $A(m,k) - D(m,k)$ satisfies $m-2i = k$.

Take any $(i,j)\in E(m,k)$.
Then $(m-2i-2j) - (m-2i) = -2j < 0$.
So $(i,j)\in C(m,k)$ and $E(m,k)$ is a subset of $C(m,k)$.
The difference of those two index sets are consists of $(i,j)\in C(m,k)$ satisfying $m-2i\geq k$.

Therefore, \[A+D=\sum_{(i,j,s,t)\in A(m,k) - D(m,k)}{s-i-1\choose i-2t}\lambda_{j-1}\lambda_{i-1}\lambda_{I}\otimes x\Sq^{s}\Sq^{t}\]
and \[E+C=\sum_{(i,j)\in C(m,k) - E(m,k)}{j-1\choose i}\lambda_{j-1}\lambda_{i-1}\lambda_{I}\otimes x\Sq^{i+j}.\]
The index $(i,j,s,t)$ in $A(m,k) - D(m,k)$ and the index $(i,j)$ in $C(m,k) - E(m,k)$ both satisfy $m-2i\geq k$.
By Lemma \ref{lem:lambda_comp_0}, $\lambda_{i-1}\lambda_{I} = 0$ when $m-2i\geq k$ and thus $A+D = 0, E+C = 0$.
Adding them up, we get $d^{2}(\lambda_{I}\otimes x) = A + D + E + C = 0$.
\end{proof}

\begin{remark}
$\Lambda_{k}$ is a contravariant exact functor from $\mathcal{U}_{k}$ to $\Ch^{*}(\Gr(\mathbb{F}_{2}\Mod))$.
\end{remark}

\subsection{Cohomology of the cochain complex $\Lambda_{k}(M)$}
The following thoerem is our main result in this subsection.
It is the correspondence of Proposition \ref{prop:lambda_coho} in the world of $\mathcal{U}_{k}$.
\begin{theorem}
\label{thm:lambda_k_coho}
For any module $M\in\mathcal{U}_{k}$, the $\Lambda_{k}(M)$ is a length $\leq k$ cochain complex with
\[H^{s,a}(\Lambda_{k}(M)) = \Ext_{k}^{s}(M,S_{k}(a))\quad\textrm{for all }s,a.\]
\end{theorem}
\begin{proof}
Let $M$ be any module in $\mathcal{U}_{k}$.
Consider a free resolution of $M$ \[\cdots\to P_{1}\to P_{0}\to M\to 0.\]
Applying the functor $\Lambda_{k}$ to the complex \[\cdots\to P_{2}\to P_{1}\to P_{0}\to 0,\] we get the following double complex $C^{*,*}$
\[\begin{tikzcd}
&\vdots&\vdots&\vdots&\\
\cdots&\Lambda_{k}^{2}(P_{2})\arrow[l]\arrow[u] &\Lambda_{k}^{2}(P_{1}) \arrow[l]\arrow[u] & \Lambda_{k}^{2}(P_{0}) \arrow[l] \arrow[u]& 0\arrow[l]\\
\cdots &\Lambda_{k}^{1}(P_{2})\arrow[l]\arrow[u] &\Lambda_{k}^{1}(P_{1}) \arrow[l]\arrow[u] & \Lambda_{k}^{1}(P_{0}) \arrow[l] \arrow[u]& 0\arrow[l]\\
\cdots &\Lambda_{k}^{0}(P_{2})\arrow[l]\arrow[u] &\Lambda_{k}^{0}(P_{1}) \arrow[l]\arrow[u] & \Lambda_{k}^{0}(P_{0}) \arrow[l] \arrow[u]& 0\arrow[l]\\
& 0\arrow[u]& 0\arrow[u]& 0\arrow[u]&
\end{tikzcd}\]
We get two spectral sequences from the double complex $C^{*,*}$ with $C^{s,i} = \Lambda_{k}^{s}(P_{i})$.
As we will see, both spectral sequences collapse at the second page.

Taking horizontal cohomology first, we get only one nontrivial column
\[\begin{tikzcd}
&\vdots&\vdots&\vdots&\\
\cdots&0\arrow[l]\arrow[u]&0\arrow[l]\arrow[u]&\Lambda_{k}^{2}(M)\arrow[l]\arrow[u]&0\arrow[l]\\
\cdots&0\arrow[l]\arrow[u]&0\arrow[l]\arrow[u]&\Lambda_{k}^{1}(M)\arrow[l]\arrow[u]&0\arrow[l]\\
\cdots&0\arrow[l]\arrow[u]&0\arrow[l]\arrow[u]&\Lambda_{k}^{0}(M)\arrow[l]\arrow[u]&0\arrow[l]\\
&0\arrow[u]&0\arrow[u]&0\arrow[u]&
\end{tikzcd}\]
beause the functor $\Lambda_{k}^{s}:\mathcal{U}_{k}^{\op}\to\mathbb{F}_{2}\Mod$ is exact.
Then taking vertical cohomology, we get $H^{s}(\Lambda_{k}(M))$ at position $(s,0)$ and zero everywhere else.
Its absolute internal degree $a$ part is $H^{s,a}(\Lambda_{k}(M))$.

If instead taking vertical cohomology first, we get only one nontrivial row according to Proposition \ref{prop:lambda_k_coho}
\[\begin{tikzcd}
&\vdots&\vdots&\\
\cdots&0 \arrow[l]\arrow[u] & 0 \arrow[l] \arrow[u]& 0\arrow[l]\\
\cdots&\bigoplus_{a\geq 0}\Hom_{k}(P_{1}, S_{k}(a)) \arrow[l]\arrow[u] & \bigoplus_{a\geq 0}\Hom_{k}(P_{0}, S_{k}(a)) \arrow[l] \arrow[u]& 0\arrow[l]\\
& 0\arrow[u]& 0\arrow[u]&
\end{tikzcd}\]
Then taking horizontal cohomology, we get $\bigoplus_{a\geq 0}\Ext_{k}^{s}(M,S_{k}(a))$ at position $(0, s)$ and zero everywhere else.
Its absolute internal degree $a$ piece is $\Ext_{k}^{s}(M,S_{k}(a))$.

Since those two limits of spectral sequences are the same and the absolute internal degrees should match, we have \[H^{s,a}(\Lambda_{k}(M)) = \Ext^{s}_{k}(M,S_{k}(a)).\qedhere\]
\end{proof}

\begin{proposition}
\label{prop:lambda_k_coho}
\[H^{s,a}(\Lambda_{k}F_{k}(n)) = 
\left\{\begin{split}
&\mathbb{F}_{2}&\textrm{if } s = 0,a = n\\
&0&\textrm{otherwise}
\end{split}\right.
\]
Or equivalently,
\[H^{s,a}(\Lambda_{k}F_{k}(n)) = 
\left\{\begin{split}
&\Hom_{k}(F_{k}(n), S_{k}(a))&\textrm{if } s = 0\\
&0&\textrm{if }s > 0
\end{split}\right.
\]
\end{proposition}
\begin{proof}
Construction \ref{cons:P^u} provides a decreasing filtration $\left(P^{u}\right)_{u\geq 0}$ of the cochain complex $\Lambda_{k}F_{k}(n)$.
Construction \ref{cons:Q^u_L} provides an increasing multi-filtration $\left(Q^{u}_{L}\right)_{L\in\mathbb{N}^{u}}$ of the cochain complex $P^{u}/P^{u+1}$.
By Lemma \ref{lem:Q^u_L_union}, their union is \[\bigcup_{L\in\mathbb{N}^{u}}Q^{u}_{L} = P^{u}/P^{u+1}.\]
Denote by $Q^{u}_{<L}$ the sum of cochain complexes $Q^{u}_{L'}$ with $L'< L$ lexicographically.
Lemma \ref{lem:hom_Q_ag} calculates the cohomology of the associated graded $Q^{u}_{L}/Q^{u}_{<L}$: it has zero cohomology when $u\geq 1$; the associated graded $Q^{0}_{\emptyset}/Q^{0}_{<\emptyset}$ has cohomology $\mathbb{F}_{2}$ at $(s,a) = (0,n)$ and zero everywhere else.
Adding them up, we get the cohomology of the original cochain complex $\Lambda_{k}F_{k}(n)$.
\end{proof}

\begin{remark}
If $\Sq^{J}\iota_{n}$ is an element of the basis of $F_{k}(n)$ by admissibles as in Proposition \ref{prop:Fkn} and is of degree $m$, then we use $\left(\Sq^{J}\iota_{n}\right)^{\vee}$ to denote the element in $F_{k}(n)_{m} = \Hom(F_{k}(n)^{m},\mathbb{F}_{2})$ which sends $\Sq^{J}\iota_{n}$ to 1 and all other admissible basis vectors to 0.
Every element in $\Lambda_{k}F_{k}(n)$ can be written uniquely as a sum of $\lambda_{I}\otimes\left(\Sq^{J}\iota_{n}\right)^{\vee}$ with
\begin{itemize}
\item
$\lambda_{I}$ and $\Sq^{J}$ are both admissible,
\item
$\lambda_{I}$ is nonzero in $\Lambda_{k}(m)$ where $m$ is the degree of $\Sq^{J}\iota_{n}$.
\end{itemize}
We say $\lambda_{I}\otimes\left(\Sq^{J}\iota_{n}\right)^{\vee}$ is \emph{admissible} if it satisfies the two conditions above.
In general, when $S$ forms an additive basis for $T$, we say $s\in S$ \emph{shows up} in $t\in T$ if $t = \sum_{s'\in S'}s'$ with $S'\subseteq S$ and $s\in S'$.
\end{remark}

\begin{construction}
\label{cons:P^u}
For any $u\geq 0$, define $P^{u}$ to be the sub-bigraded vector space of $\Lambda_{k}F_{k}(n)$ spanned by admissible $\lambda_{I}\otimes\left(\Sq^{J}\iota_{n}\right)^{\vee}$ with $|I| + |J| \geq u$.
By $|I|$ and $|J|$ we mean the lengths of them.
\end{construction}

The $P^{u}$ is a subcomplex of $\Lambda_{k}F_{k}(n)$ by virtue of the following lemma.

\begin{lemma}
$P^{u}$ is closed under the differential.
\end{lemma}
\begin{proof}
We need to prove $d(P^{u})\subseteq P^{u}$.
Let $\lambda_{I}\otimes\left(\Sq^{J}\iota_{n}\right)^{\vee}$ be admissible.
We have \[d(\lambda_{I}\otimes\left(\Sq^{J}\iota_{n}\right)^{\vee}) = \sum_{I',J'}\lambda_{I'}\otimes\left(\Sq^{J'}\iota_{n}\right)^{\vee}.\]
It suffices to prove $|I'| + |J'| \geq |I| +|J|$.
We know $|I'| = |I| + 1$.
So it suffices to prove $|J'|\geq |J| - 1$.
If $J' = J$, we are done.
Otherwise, $\lambda_{I'}\otimes\left(Sq^{J'}\iota_{n}\right)^{\vee}$ must show up in $\lambda_{i-1}\lambda_{I}\otimes\left(\Sq^{J}\iota_{n}\right)^{\vee}\Sq^{i}$, which means that $\Sq^{J}\iota_{n}$ must show up in $\Sq^{i}\Sq^{J'}\iota_{n}$.
The Adem relations do not increase the length of the Steenrod squares.
So $|J|\leq 1 + |J'|$.
\end{proof}

Observe that $P^{u+1}$ is a subcomplex of $P^{u}$.

\begin{construction}
\label{cons:Q^u_L}
The associated graded $P^{u}/P^{u+1}$ is a cochain complex spanned by the admissible $\lambda_{I}\otimes\left(\Sq^{J}\iota_{n}\right)^{\vee}$ with $|I| + |J| = u$.
For any $u\geq 0$ and any $L\in\mathbb{N}^{u}$, define $Q_{L}^{u}$ to be the sub-bigraded vector space of $P^{u}/P^{u+1}$ spanned by admissible $\lambda_{I}\otimes\left(\Sq_{J}\iota_{n}\right)^{\vee}$ with \[(I(s) + 1, \ldots, I(1) + 1, J(1), \ldots,J(t))\leq L,\] where $s:=|I|, t:=|J|$.
The index set $\mathbb{N}^{u}$ is ordered lexigraphically.
For example, when $u = 2$ and $L \in \mathbb{N}^{2}$, we have $(1,6) < (4,1)$ and $(1, 2) < (1, 4)$.
When $u = 0$, we have $\mathbb{N}^{0} = \{\emptyset\}$ and $Q_{\emptyset}^{0}$ is spanned by $1\otimes\iota_{n}$.
\end{construction}

The $Q_{L}^{u}$ is a subcomplex of $P^{u}/P^{u+1}$ by virtue of the following lemma.

\begin{lemma}
\label{lem:Q^u_L}
$Q^u_L$ is closed under the differential.
\end{lemma}
\begin{proof}
When $u = 0$ and $L = \emptyset$, $Q^{0}_{\emptyset}$ is a cochain complex with only one nontrivial element $1\otimes\iota_{n}$.
Assume $u\geq 1$.
We need to prove $d(Q^{u}_{L})\subseteq Q^{u}_{L}$.
Let $\lambda_{I}\otimes\left(\Sq^{J}\iota_{n}\right)^{\vee}$ be admissible in $P^{u}/P^{u+1}$.
We have \[d(\lambda_{I}\otimes\left(\Sq^{J}\iota_{n}\right)^{\vee}) = \sum_{I',J'}\lambda_{I'}\otimes\left(\Sq^{J'}\iota_{n}\right)^{\vee}.\]
It suffices to prove \[(I'(s+1)+1,\ldots,I'(1)+1,J')\leq (I(s)+1,\ldots,I(1)+1,J).\]
Since everything is in $P^{u}/P^{u+1}$, we have $|I'| + |J'| = |I| + |J|$, so $|I'| = |I| + 1$ and $|J'| = |J| - 1$.
So $\lambda_{I'}\otimes\left(Sq^{J'}\iota_{n}\right)^{\vee}$ shows up in $\lambda_{i-1}\lambda_{I}\otimes\left(\Sq^{J}\iota_{n}\right)^{\vee}\Sq^{i}$.
That is, $\lambda_{I'}$ shows up in $\lambda_{i-1}\lambda_{I}$, and $\Sq^{J}\iota_{n}$ shows up in $\Sq^{i}\Sq^{J'}\iota_{n}$.
The Adem relations in the Steenrod algebra tell us that $J\geq (i, J')$.
So \[(I(s)+1,\ldots,I(1) + 1, i, J')\le (I(s)+1,\ldots,I(1) + 1,J).\]
If $\lambda_{i-1}\lambda_{I}$ is not admissible, then \[(I'(s+1) + 1, \ldots, I'(1) + 1) < (I(s) + 1, \ldots, I(1) + 1, i).\]
If $\lambda_{i-1}\lambda_{I}$ is admissible, then \[(I'(s+1) + 1, \ldots, I'(1) + 1) = (I(s) + 1, \ldots, I(1) + 1, i).\]
So in either case, we have \[(I'(s+1)+1,\ldots,I'(1)+1,J')\leq (I(s)+1,\ldots,I(1)+1,J).\qedhere\]
\end{proof}

Observe that $Q^{u}_{L}$ is a subcomplex of $Q^{u}_{L'}$ when $L\leq L'$ lexigraphically.

\begin{lemma}
\label{lem:Q^u_L_union}
\[\bigcup_{L\in\mathbb{N}^{u}}Q^{u}_{L} = P^{u}/P^{u+1}.\]
\end{lemma}
\begin{proof}
When $u = 0$, we have $Q^{0}_{\emptyset} = P^{0}/P^{1}$.
When $u\geq 1$, $\lambda_{I}\otimes\left(\Sq^{J}\iota_{n}\right)^{\vee}$ lives in $Q^{u}_{L}$ when $L = (I(s)+1,\ldots,I(1) + 1, J)$.
\end{proof}

\begin{lemma}
\label{lem:hom_Q_ag}
Define $Q^{u}_{<L}$ as the sum of cochain complexes $Q^{u}_{L'}$ with $L'< L$.
The associated graded $Q^{u}_{L}/Q^{u}_{<L}$ has zero cohomology when $u\geq 1$.
The associated graded $Q^{0}_{\emptyset}/Q^{0}_{<\emptyset}$ has cohomology $\mathbb{F}_{2}$ at $(s,a) = (0,n)$ and zero everywhere else.
\end{lemma}
\begin{proof}
The case $u = 0$ is obvious.
Assume $u\geq 1$.
The associated graded $Q^{u}_{L}/Q^{u}_{<L}$ is spanned by admissible $\lambda_{I}\otimes(\Sq^{J}\iota_{n})^{\vee}$ with \begin{equation}\label{eq:lem:hom_Q_ag}(I(s)+1,\ldots,I(1) + 1, J(1),\ldots,J(t)) = L.\end{equation}
When there exsit no admissible $\lambda_{I}$ and $\Sq^{J}$ such that equation (\ref{eq:lem:hom_Q_ag}) is true, the associated graded will be zero and we are done.
Now assume there exist admissible $\lambda_{I}$ and $\Sq^{J}$ such that equation (\ref{eq:lem:hom_Q_ag}) is true.

The admissibility requires that $I(r+1)\leq 2I(r)$ or equivalently $I(r+1)+1<2(I(r)+1)$ and $J(r)\ge 2J(r+1)$.
So there are two admissibles $\lambda_{I}$ and $\Sq^{J}$ such that equation (\ref{eq:lem:hom_Q_ag}) is true.
They are $\lambda_{I}\otimes\left(\Sq^{J}\iota_{n}\right)^{\vee}$ and $\lambda_{J(1) - 1}\lambda_{I}\otimes\left(\Sq^{J(2)}\cdots\Sq^{J(t)}\iota_{n}
\right)^{\vee}$ with $2(J(1)-1)\geq I(1)$.
Note that $\lambda_{I} = 0$ in $\Lambda_{k}(m)$ if and only if $\lambda_{J(1) - 1}\lambda_{I} = 0$ in $\Lambda_{k}(m-J(1))$.
So the associated graded will either be zero or $\cdots\to0\to\mathbb{F}_{2}\to\mathbb{F}_{2}\to0\to\cdots$.

According to the proof of Lemma \ref{lem:Q^u_L}, the differential will be
\[d\left(\lambda_{I}\otimes\left(\Sq^{J}\iota_{n}\right)^{\vee}\right) = \lambda_{J(1) - 1}\lambda_{I}\otimes\left(\Sq^{J(2)}\cdots\Sq^{J(t)}\iota_{n}\right)^{\vee}\] if $I(1)\leq 2(J(1) - 1)$ and zero otherwise.
Therefore, the map $\mathbb{F}_{2}\to\mathbb{F}_{2}$ is the identity map.
So the cohomology of the associated graded is always zero when $u\geq 1$.
\end{proof}

\begin{remark}
The arguments in this subsection give another proof to Proposition \ref{prop:lambda_coho}.
\end{remark}

\section{Inverse system of $\Ext$ groups}
\label{sec:invsys}
Since the forgetful functor $u:\mathcal{U}_{k+1}\to\mathcal{U}_{k}$ preserves projectives by Proposition \ref{prop:prefor} and is exact by Proposition \ref{prop:exfor}, it induces a map of $\Ext$ groups $\Ext_{k+1}^{s}(M, N)\to \Ext_{k}^{s}(uM, uN)$, where $M$ and $N$ are any two modules in $\mathcal{U}_{k+1}$.
Therefore, given any two modules $M$ and $N$ in $\mathcal{U}$, we have an inverse system of $\Ext$ groups
\[\cdots\to\Ext^{s}_{2}(uM, uN)\to \Ext^{s}_{1}(uM, uN)\to \Ext^{s}_{0}(uM, uN).\]
In this section, we will study this inverse system and its inverse limit.
The main result of this section is summarized in the following theorem.

\begin{theorem}
\label{thm:is}
Let $M$ and $N$ be two nonzero modules in the category $\mathcal{U}$.
Let $s$ be any nonnegative integer.
If $N$ is bounded above with top nontrivial degree $n$, then the inverse system
\[\cdots\to\Ext^{s}_{2}(uM, uN)\to \Ext^{s}_{1}(uM, uN)\to \Ext^{s}_{0}(uM, uN).\]
stablizes and the limit is equal to $\Ext^{s}_{\mathcal{U}}(M,N)$.
More specifically, when $k\geq n-1$, the maps
\[\Ext_{k+1}^{s}(uM,uN)\to\Ext_{k}^{s}(uM,uN)\]
and
\[\Ext_{\mathcal{U}}^{s}(M,N)\to\Ext_{k}^{s}(uM,uN)\]
are isomorphisms.
\end{theorem}
Before proving this theorem, we take some time defining the decomposables and indecomposables of modules in $\mathcal{U}$ or $\mathcal{U}_{k}$ and presenting several lemmas.

\begin{definition}[Decomposables and indecomposables]
Let $M$ be a module in $\mathcal{U}$.
Let $N$ be the submodule of $M$ generated by $\Sq^{i}x$ with $i\geq 1$ and $x\in M$.
The submodule $N$ is called the decomposables of $M$ and the quotient module $M/N$ is called the indecomposables of $M$.
Note that both $N$ and $M/N$ live in $\mathcal{U}$.
\end{definition}

Similar definitions apply for a module in $\mathcal{U}_{k}$.

\begin{definition}[Decomposables and indecomposables]
Let $k\geq 0$ and $M$ be a module in $\mathcal{U}_{k}$.
Let $N$ be the submodule of $M$ generated by $\Sq^{i}x$ where $i\geq 1, |x| - i < k$ and $x$ is any nonzero homogeneous element in $M$.
The submodule $N$ is called the decomposables of $M$ and the quotient module $M/N$ is called the indecomposables of $M$.
Note that both $N$ and $M/N$ live in $\mathcal{U}_{k}$.
\end{definition}

\begin{lemma}
\label{lem:sphere_indecom}
If $M$ is a module in $\mathcal{U}_{k}$, then $\mathcal{U}_{k}(M, S_{k}(n))$ is equal to the dual of the degree $n$ part of the indecomposables of $M$.
If $M$ is a module in $\mathcal{U}$, then $\mathcal{U}(M,S(n))$ is equal to the dual of the degree $n$ part of the indecomposables of $M$.
\end{lemma}
\begin{proof}
We assume $M$ is a module in $\mathcal{U}_{k}$.
The $\mathcal{U}$ case is quite similar and omitted.
The $\mathbb{F}_{2}$-module $\mathcal{U}_{k}(M, S_{k}(n))$ is equal to the set of $\mathbb{F}_{2}$-linear maps from $M^{n}$ to $\mathbb{F}_{2}$ such that \[f(\Sq^{i}x) = 0\quad\forall i\geq 1, n-2i<k, x\in M^{n-i}.\]
It is equivalent to the the set of $\mathbb{F}_{2}$-linear maps from $(M/N)^{n}$ to $\mathbb{F}_{2}$, where $N$ is the decomposables of $M$.
\end{proof}

\begin{lemma}
\label{lem:Ext_iso_sphere}
Suppose that $M$ is a module in $\mathcal{U}_{k+1}$.
If $n\not\in\{k+2,k+4,k+6,\ldots\}$, then $\Ext_{k+1}^{s}(M, S_{k+1}(n))\to \Ext_{k}^{s}(uM, S_{k}(n))$ is an isomorphism for any $s$.
\end{lemma}
\begin{proof}
It suffices to prove that $\mathcal{U}_{k+1}(M, S_{k+1}(n))\to \mathcal{U}_{k}(uM, S_{k}(n))$ is an isomorphism.
By Lemma \ref{lem:sphere_indecom}, it suffices to prove that the degree $n$ part of the indecomposables of $M$ and $uM$ are the same.
Those two indecomposables only differ by the image of the lower Steenrod operations $\Sq_{k} = \Sq^{m-k}$ from degree $m$ to degree $2m-k$ with $m\geq k+1$.
But those $\Sq_{k}$'s does not have target degree $n$ because its target degrees are $k+2, k+4,k+6,\ldots$.
\end{proof}

\begin{lemma}
\label{lem:Ext_iso_single_deg}
Suppose that $M, N$ are two modules in $\mathcal{U}_{k+1}$.
If $n\not\in\{k+2,k+4,k+6,\ldots\}$ and $N^{i} = 0$ for any $i\neq n$, then $\Ext_{k+1}^{s}(M, N)\to \Ext_{k}^{s}(uM, uN)$ is an isomorphism for any $s$.
\end{lemma}
\begin{proof}
The module $N$ is equal to the a direct sum of some sphere modules $S_{k+1}(n)$.
Say $N = \bigoplus_{j\in J} S_{k+1}(n)$, where $J$ is an index set.
Therefore, \[\Ext_{k+1}^{s}(M, N) = \bigoplus_{j\in J}\Ext_{k+1}^{s}(M, S_{k+1}(n))\] and \[uN = \bigoplus_{j\in J} S_{k}(n),\quad\Ext_{k}^{s}(M, N) = \bigoplus_{j\in J}\Ext_{k}^{s}(M, S_{k}(n)).\]The map $\Ext_{k+1}^{s}(M, S_{k+1}(n))\to \Ext_{k}^{s}(M, S_{k}(n))$ is an isomorphism by Lemma \ref{lem:Ext_iso_sphere} and thus so is the map $\Ext_{k+1}^{s}(M, N)\to \Ext_{k}^{s}(uM, uN)$.
\end{proof}

\begin{proposition}
\label{prop:U_k_is}
Suppose that $M, N$ are two modules in $\mathcal{U}_{k+1}$.
If $N$ is bounded above with the top nontrivial degree $n \leq k+1$, then $\Ext_{k+1}^{s}(M, N)\to \Ext_{k}^{s}(uM, uN)$ is an isomorphism for any $s$.
\end{proposition}
\begin{proof}
We will proceed by induction on $n$.
The lemma above solves the base case $n = 0$.
Now assume $n > 0$.
We have the following short exact sequence in $\mathcal{U}_{k+1}$
\[0\to N'\to N\to N''\to 0,\]
where $N'$ is the degree $n$ part of $N$ and $N''$ is the degree $<n$ part of $N$.
This short exact sequence gives rise to a long exact of $\Ext$ groups
\[\cdots\to \Ext_{k+1}^{s}(M, N')\to \Ext_{k+1}^{s}(M, N)\to \Ext_{k+1}^{s}(M, N'')\to \cdots\]
Since the forgetful functor $u:\mathcal{U}_{k+1}\to\mathcal{U}_{k}$ is exact, we have the following short exact sequence in $\mathcal{U}_{k}$
\[0\to uN'\to uN\to uN''\to 0\] and similarly we have the following long exact sequence of $\Ext$ groups
\[\cdots\to \Ext_{k}^{s}(uM, uN')\to \Ext_{k}^{s}(uM, uN)\to \Ext_{k}^{s}(uM, uN'')\to \cdots\]
The two long exact sequences above form a commutative diagram
\[\begin{tikzcd}
\cdots\arrow[r]&\Ext_{k+1}^{s}(M, N')\arrow[r]\arrow[d]&\Ext_{k+1}^{s}(M, N)\arrow[r]\arrow[d] &\Ext_{k+1}^{s}(M, N'')\arrow[r]\arrow[d]&\cdots\\
\cdots\arrow[r]&\Ext_{k}^{s}(uM,uN')\arrow[r]&\Ext_{k}^{s}(uM,uN)\arrow[r]&\Ext_{k}^{s}(uM, uN'')\arrow[r]&\cdots
\end{tikzcd}\]
By Lemma \ref{lem:Ext_iso_single_deg}, the map $\Ext^{s}_{k+1}(M,N')\to\Ext^{s}_{k}(uM,uN')$ is an isomorphism for any $s$.
By the induction hypothesis, the map $\Ext^{s}_{k+1}(M,N'')\to\Ext^{s}_{k}(uM,uN'')$ is an isomorphism for any $s$.
By five lemma, the middle map $\Ext^{s}_{k+1}(M, N)\to\Ext^{s}_{k}(uM,uN)$ is an isomorphism for all $s$.
\end{proof}

\begin{lemma}
\label{lem:Ext_iso_single_deg_U}
Suppose that $M,N$ are two modules in $\mathcal{U}$.
If $n \leq k+1$ and $N^{i} = 0$ for any $i\neq n$, then the map $\Ext_{\mathcal{U}}^{s}(M,N)\to\Ext_{k}^{s}(uM,uN)$ is an isomorphism for all $s$.
\end{lemma}
\begin{proof}
The proof is similar to the proof to Lemma \ref{lem:Ext_iso_single_deg}.
The core observation is that when $n \leq k+1$, the degree $n$ parts of the indecomposables of $M\in\mathcal{U}$ and $uM\in\mathcal{U}_{k}$ are the same.
\end{proof}

\begin{proposition}
\label{prop:U_is}
Suppose that $M,N$ are two modules in $\mathcal{U}$.
If $N$ is bounded above and the top nontrivial degree $n \leq k+1$, then $\Ext_{\mathcal{U}}^{s}(M,N)\to\Ext_{k}^{s}(uM,uN)$ is an isomorphism for any $s$.
\end{proposition}
\begin{proof}
The proof is quite similar to the proof to Propositon \ref{prop:U_k_is}.
So we only give a sketch.
We split $N$ into the degree $n$ part $N'$ and the degree $<n$ part $N''$.
Then  Lemma \ref{lem:Ext_iso_single_deg_U} and an induction on $n$ complete the proof.
\end{proof}

Proposition \ref{prop:U_k_is} and Proposition \ref{prop:U_is} together give a proof to Theorem \ref{thm:is}.

\bibliography{ref}
\bibliographystyle{alpha}

\end{document}